\documentclass[10pt,letterpaper]{article}
\usepackage[T1]{fontenc}
\usepackage[utf8]{inputenc}
\usepackage{amsmath}
\usepackage{amsfonts}
\usepackage{amssymb}
\usepackage{graphicx}
\usepackage{verbatim}
\usepackage{mathrsfs}
\usepackage{upref,amsthm,amsxtra,exscale}
\usepackage{dsfont}
\usepackage{nameref}
\usepackage{varioref}
\usepackage[colorlinks=true,urlcolor=blue,
citecolor=red,linkcolor=blue,linktocpage,pdfpagelabels,
bookmarksnumbered,bookmarksopen]{hyperref}
\usepackage{cleveref}
\usepackage{cite}
\usepackage{fullpage}
\usepackage[dvipsnames]{xcolor}
\usepackage{bbm}
\usepackage{appendix}

\usepackage[dvipsnames]{xcolor}

\newtheorem{theorem}{Theorem}[section]
\newtheorem{corollary}[theorem]{Corollary}
\newtheorem{remark}[theorem]{Remark}

\newtheorem{lemma}[theorem]{Lemma}
\newtheorem{proposition}[theorem]{Proposition}
\newtheorem{definition}[theorem]{Definition}
\newtheorem{example}[theorem]{Example}

\numberwithin{equation}{section}

\def\r{\mathbb{R}}
\def\rn{\mathbb{R}^N}
\def\z{\mathbb{Z}}

\def\s1{\mathbb{S}^1}
\def\n{\mathbb{N}}
\def\cc{\mathbb{C}}
\def\eps{\varepsilon}
\def\rh{\rightharpoonup}
\def\io{\int_{\Omega}}
\def\irn{\int_{\r^N}}
\def\vp{\varphi}

\def\vr{\varrho}
\def\o{\Omega}
\def\t{\Theta}

\def\cC{\mathcal{C}}

\def\cI{\mathcal{I}}
\def\cJ{\mathcal{J}}

\def\supp{\text{supp}}
\def\bar{\overline}

\def\d{\,\mathrm{d}}
\def\e{\mathrm{e}}
\def\i{\mathrm{i}}
\def\dist{\mathrm{dist}}

\def\sgn{\operatorname{sign}}
\def\supp{\operatorname{supp}}

\newcommand{\norm}[1]{\left \| #1\right \|}
\newcommand{\R}{\mathbb{R}}

\definecolor{DarkGreen}{rgb}{0.0, 0.5, 0.0}

\title{Sublinear elliptic equations with a sharp change of sign in the nonlinearity}
\author{Mónica Clapp, Alberto Saldaña, and Delia Schiera}
\date{}

\begin{document}

\maketitle

\begin{abstract}
We study the semilinear indefinite elliptic problem
\[
-\Delta u = Q_\Omega |u|^{p-2}u \quad \text{in } \mathbb{R}^N,
\]
where $Q_\Omega = \chi_\Omega - \chi_{\mathbb{R}^N \setminus \Omega}$, $\Omega \subset \mathbb{R}^N$ is a bounded smooth subset, $N \geq 3$, and $1 \leq p < 2$, with $p=1$ corresponding to the sign nonlinearity. Using a variational approach, we investigate the uniqueness or multiplicity of nonnegative solutions depending on the shape of $\Omega$ and the existence of different types of nodal solutions.  We also show that all solutions have compact support and analyze how the support of the ground state depends on $p$, 
proving convergence to the whole space as $p\to 2^{-}$ and identifying some 
qualitative features such as starshapedness and Lipschitz regularity of the support. We also establish a link between these problems and a two-phase Serrin-type torsion overdetermined problem. 

\medskip

\noindent \textbf{Mathematics Subject Classification:} 
35J61	%Semilinear elliptic equations
35A02   %Uniqueness problems for PDEs: global uniqueness, local uniqueness, non-uniqueness
(primary) 
35J20   %Variational methods for second-order elliptic equations
35B09   %Positive solutions to PDEs
35B40 %Asymptotic behavior of solutions to PDEs
(secondary).

\medskip

\noindent \textbf{Keywords:} 
Uniqueness and multiplicity of nonnegative solutions,
mountain pass solutions,
sublinear indefinite problems, 
dead core solutions, starshaped support.
\end{abstract}

\section{Introduction}

Let $\o$ be a smooth open bounded subset of $\rn$ (not necessarily connected), $N\geq 3$, and let 
\begin{equation*}
Q_\o(x):=
\begin{cases}
1, & \text{if \ }x\in\o, \\
-1,& \text{if \ }x\in\rn\smallsetminus\o.
\end{cases}
\end{equation*}	
In this paper we study nonnegative and nodal solutions of the following problem
\begin{align}\label{sublinear}
    -\Delta u=Q_\o \, f_p(u),\qquad u \in X^p:=D^{1,2}(\rn)\cap L^{p}(\rn),
\end{align}
where 
$D^{1,2}(\rn):=\{u\in L^{2^*}(\rn)\;:\;\nabla u\in L^2(\rn,\rn)\}$,  
\begin{align}\label{fp}
f_p(t):=|t|^{p-2} t\quad \text{ for } p\in(1,2)\qquad \text{ and }\qquad  f_1(t):=\sgn(t),
\end{align}
where $\sgn$ denotes the sign nonlinearity
\begin{align*}
\sgn(t):=\begin{cases}
    1, & \text{ if } t >0,\\
    0,& \text{ if } t=0,\\
    -1, & \text{ if } t<0. 
\end{cases} 
\end{align*} 

Problem \eqref{sublinear} is related to some models in nonlinear optics that study waveguides propagating through a stratified dielectric medium, see \cite{stuart, stuart2}. It has also been studied in the context of population biology \cite{spruck,GURTIN} and its system counterpart has been employed to model competing species attracted to a specific region in space and repelled from its complement; see \cite{CSS}.

This motivates the study of solutions to \eqref{sublinear}, which has been investigated recently using variational methods, where the exponent $p$ plays a crucial role since, depending on its value, the existence and the qualitative properties of the solutions change. The critical case $f_{p}(t):=|t|^{p-2} t$ with $p=2^*=\frac{2N}{N-2}$ is studied in \cite{CPS25,CFS25}, where existence and nonnexistence of solutions is guaranteed depending on the topology and the geometry of the set $\Omega.$ The superlinear and linear cases $p\in [2,2^*)$ are investigated in \cite{CHS25,CMSS}, where existence of ground states (i.e. least energy solutions), least energy nodal solutions, and eigenfunctions is shown together with symmetry and sharp decay estimates. In particular, these decay estimates give a clear example on the decisive role that the exponent $p$ has in the qualitative behavior of solutions.  For instance, if $u$ is a solution of \eqref{sublinear}, then there is a constant $C>0$ such that, for $|x|>1$,
\begin{align*}
    |u(x)|&< C |x|^{2-N}, \hspace{2.1cm} \text{ if }p\in \left(\frac{2N-2}{N-2},2^*\right),\\
    |u(x)|&< C(|x|\sqrt{\ln|x|})^{2-N}, \qquad \text{ if }p=\frac{2N-2}{N-2},\\
    |u(x)|&< C|x|^{\frac{2}{2-p}}, \hspace{2.2cm} \text{ if }p\in \left(2,\frac{2N-2}{N-2}\right),\\
    |u(x)|&< C|x|^{\frac{1-N}{2}}e^{-\sqrt{\Lambda_1}|x|}, \qquad \ \text{ if }p=2,
\end{align*}
where $\Lambda_1:=\inf \left\{\int_{\R^N} |\nabla v|^2\::\: v \in D^{1,2}(\R^N),\ \int_{\R^N} Q_\Omega |v|^2=1\right\}>0.$  As can be seen, the decay of the solutions improves as $p$ diminishes. It is natural to ask what is the decay of solutions for $p\in [1,2)$. But, what is better than exponential decay for a solution of an elliptic equation?  In this paper, we show that solutions in the sublinear regime have \emph{compact support}.  Furthermore, we give a broad overview of the existence, uniqueness, and other qualitative properties of the solutions of \eqref{sublinear} for $p\in[1,2).$

Our first result describes the existence of a unique ground state and of two nodal solutions: one of least-energy type and one of mountain pass type. It also states that \emph{all} solutions have compact support and that the support grows (at least in the case of the ground state) as $p$ tends to the eigenvalue case $p=2$.   

\begin{theorem}\label{thm:main}
Let $p\in[1,2)$ and let $\o$ be a smooth open bounded subset of $\rn$. 
\begin{enumerate}
    \item \eqref{sublinear} has a unique nonnegative ground state $w_p$, and $w_p>0$ on $\overline \Omega$.
    \item All solutions of \eqref{sublinear} have compact support. Furthermore, the support of the nonnegative ground state $w_p$ tends to $\rn$ as $p\to 2^-$ and it is asymptotically spherical. More precisely, there is $p_0=p_0(\Omega)\in (1,2)$ so that, for every $p\in (p_0,2)$, there is $r_p>0$ satisfying that 
\begin{align*}
\partial \{w_p>0\}\subset \{x\in \R^N\::\: r_p\leq |x|\leq r_p+2R\},
\end{align*}
 where $R>0$ is such that $\Omega\subset B_R(0)$ and $r_p\to \infty$ as $p\to 2.$
    \item If $\Omega$ is connected, then \eqref{sublinear} has a least energy nodal solution and a nodal solution of mountain pass type. Moreover, these solutions are different in some dumbbell domains.
\end{enumerate}
\end{theorem}

The existence of a ground state is obtained by global minimization. On the other hand, the existence of nodal solutions is more delicate and cannot be argued as in the superlinear case. For this, we follow two strategies.  One is to use symmetries to obtain the existence of equivariant sign changing solutions (see Theorem~\ref{thm:main:3} below), the second is to adapt the approach in \cite{BMPTW}, where a sublinear Dirichlet Lane-Emden problem is studied. For this, we show first the existence of a nodal solution of mountain pass type and, once that it is established that the set of nodal solutions is nonempty, we prove that there is one solution that has least energy among nodal solutions.  Furthermore, again following \cite{BMPTW}, we show that these two nodal solutions are different if $\Omega$ is a suitable dumbbell domain. We conjecture that, if $\Omega$ is a ball, then these two nodal solutions should coincide, however this is left as an open problem. 

The fact that all solutions have compact support follows from a comparison argument together with the fact that sublinear problems admit \emph{dead core} solutions that can be used as barriers, see Lemma~\ref{compact:sup:lem}. To see that the support of the ground state grows as $p\to 2^-$ we use an asymptotic analysis of the behavior of the solution of an auxiliary nonlinear eigenvalue problem, see \cite[Theorem 3.4]{CMSS} and Theorem~\ref{thm:supp} below.

Point 2 in Theorem \ref{thm:main} shows that the support of the ground state $w_p$ resembles an expanding ball (independently of the geometry of $\Omega$) as $p \to 2^{-}$. This asymptotic characterization is based on a monotonicity argument based on polarization invariance, see Proposition \ref{propball}.  We thank Tobias Weth for suggesting this approach to us.

The compactness of the support of solutions has been studied in other indefinite sublinear problems, see, for instance, \cite{Schatzman, spruck, AL, BPT, BPT2, BR25} and the references therein. However, we believe that problem \eqref{sublinear} presents an interesting new perspective on this kind of problem, where a more transparent and precise analysis can be done in terms of the topology and the geometry of $\Omega$. A general question that captures this idea is the following 
\begin{center}
    What is the influence that $p$ and $\Omega$ have on the shape of the support of the ground state?
\end{center}

In this direction, we have the following result showing that starshapedness of $\Omega$ forces the support to be starshaped too, and if $\Omega$ is strictly starshaped, then the free boundary is Lipschitz.
\begin{theorem}\label{thm:starshaped:intro}
    Let $p\in(1,2)$ and let $\Omega\subset \rn$ be a smooth open bounded subset of $\rn$, let $w$ be the nonnegative ground state of \eqref{sublinear} and let $K$ denote its support.
    \begin{enumerate}
        \item If $\Omega$ is starshaped, then $K$ is starshaped.
        \item If $\Omega$ is strictly starshaped, then $\partial K$ is locally the graph of a Lipschitz function.
    \end{enumerate}
\end{theorem}
The proof of Theorem~\ref{thm:starshaped:intro} relies on a Bernstein-type argument inspired by \cite{spruck}. See also Remark \ref{altphil} for further regularity results on $\partial K$ in the context of the Alt–Phillips problem (we thank Giorgio Tortone for pointing out these results to us).

We emphasize that not all sublinear problems admit solutions with compact support; see, for example, \cite{SW18}, where a unique continuation principle is shown for some sublinear problems.  This highlights the crucial role of the condition $Q_\Omega=-1$ in $\mathbb R^N\backslash \Omega$ to guarantee the existence of dead core solutions for \eqref{sublinear}.  

Theorem~\ref{thm:main} states that, for every smooth open bounded subset $\Omega$, the nonnegative ground state is unique. This uniqueness does not hold in general for nonnegative solutions, this depends on the connectedness of $\Omega$, on the distance between the connected components, and on the value of the exponent $p$.  Our next result describes the multiplicity or uniqueness of nonnegative solutions in this setting.  We use $\operatorname{Isom}(\rn)$ to denote the group of all isometries of $\rn$ (linear isometries, translations, and their compositions). In the following, nonnegative solutions are always meant to be nontrivial.

\begin{theorem}\label{thm:main:2}
    Let $p\in[1,2)$ and let $\Omega\subset \rn$ be a smooth open bounded subset of $\rn$.
    \begin{enumerate}
        \item If $\Omega$ is connected, then \eqref{sublinear} has a unique nonnegative solution given by the ground state.
\item If $\Omega\subset \rn$ is not connected, then there is $p_0=p_0(\Omega)\in(1,2)$ so that \eqref{sublinear} has a unique nonnegative solution if $p\in (p_0,2)$.
\item Let $\theta_i\subset \rn$ be a connected smooth open bounded subset for $i=1,\ldots,\ell,$ $\ell\in\mathbb N$, and let
\begin{align*}
\Omega:=\bigcup_{i=1}^\ell \omega_i,\qquad \text{ where $\omega_i:=g_i\theta_i$ for some $g_i\in \operatorname{Isom}(\rn)$.} 
\end{align*}
There is $d=d(p,\theta_1,\ldots,\theta_\ell)>0$ so that, if $\dist(\omega_i,\omega_j)>d$ for $i\neq j$, then \eqref{sublinear} has exactly $2^\ell-1$ different nonnegative solutions.
    \end{enumerate}
\end{theorem}

The proof of Theorem~\ref{thm:main:2} is based on some recent asymptotic analysis results as $p\to 2$ obtained in \cite{CMSS}, on hidden convexity arguments (see, for instance, \cite{BFMST18}), and on the use of a generalized Picone's identity \cite{BF14} as in \cite{BR25}. For the last claim in Theorem~\ref{thm:main:2}, we analyze carefully the supports of nonnegative solutions. We also note that, under these conditions, the construction of nodal solutions having the same energy as the ground state follows easily from the proof (see the proof of Theorem~\ref{thm:many}). 

We also show the existence of symmetric nodal solutions when 
$\Omega$ is symmetric via the principle of symmetric criticality. The following is a direct consequence of Theorem~\ref{thm:symm}, where more details on the shape of the solutions are given. 

\begin{theorem}\label{thm:main:3}
If $\o$ satisfies
\begin{equation*}
(\e^{\mathrm{i}\theta}z,y)\in\o\qquad\text{for every \ }\theta\in [0,2\pi], \ (z,y)\in\o, \ \text{ \ where \ }z\in\cc, \ y\in\r^{N-2},
\end{equation*}
then the problem \eqref{sublinear} has a sequence $(v_k)$ of sign-changing solutions such that $v_k\to 0$ strongly in $X^p$.
\end{theorem}

We note that Theorems~\ref{thm:main},~\ref{thm:main:2}, and~\ref{thm:main:3} also hold for $p=1$, namely, for the sign nonlinearity.  This setting is particularly challenging, because the associated energy functional is not of class $\cC^1$.   As a consequence, some tools from nonsmooth critical point theory are needed to define what is meant by critical point or by Palais-Smale sequence. However, the case $p=1$ is especially interesting, in part because it allows the computation of explicit solutions in the radial case (see Remark~\ref{sec:exp}).

Observe that, if $w$ is a ground state of \eqref{sublinear} with $p=1$, then integrating by parts \eqref{sublinear} yields the following compatibility condition
\begin{align*}
    |\{w>0\}| = 2 |\Omega|.
\end{align*}
This implies, for instance, that $\{w>0\}$ is not a ball if $\Omega$ is a thin annulus or an eccentric ellipse. 

An interesting point of view of \eqref{sublinear} for the case $p=1$ comes from the setting of Serrin-type overdetermined torsion problems.  We refer to Section~\ref{sec:over} for a broader discussion in this regard. 

\medskip

To close this introduction, we state some open questions that we find interesting.
\smallskip 

\textbf{Open Questions:}
\begin{enumerate}
    \item Is it true that the least energy nodal solution and the mountain pass type nodal solutions are the same if $\Omega$ is a ball?
   \item What is the asymptotic behavior of the mountain pass type nodal solution and of the least energy nodal solution as $p\to 2^-$?  Do they both converge to the second eigenfunction?
\end{enumerate}

Regarding the Open Question 1, we mention that this is an analogue of an open question in \cite{BMPTW}, where sublinear Dirichlet problems where studied.

\medskip

The paper is organized as follows. In Section~\ref{sec:vs} we introduce the variational frameworks and the notions of solution, including the nonsmooth setting required for the case $p=1$.  Section~\ref{sec:gs} is devoted to nonnegative ground states and nonnegative solutions, existence, uniqueness, positivity in $\Omega$, and the asymptotic behavior of nonnegative ground states as $p\to2^{-}$. Here we also show that all solutions have compact support via a comparison argument.  Section~\ref{sec:gs} also addresses the multiplicity or uniqueness of nonnegative solutions when $\Omega$ has several connected components and it contains the proof of Theorem~\ref{thm:starshaped:intro}. 
In Section~\ref{sec:nod} we construct nodal solutions, including least energy nodal solutions, mountain-pass type nodal solutions, and symmetric nodal solutions.  
In Section~\ref{sec:proofs} we give the proof of the main results stated in the Introduction. 
In Section~\ref{sec:over} we discuss the connection of \eqref{sublinear} with overdetermined torsion-type problems when $p=1$. 

\section{The variational setting}\label{sec:vs}

Let $p \in [1, 2)$.  We endow the Banach space $X^p:=D^{1,2}(\rn)\cap L^{p}(\rn)$ with the norm
\begin{align*}
   \|u\|_{X^p}:=\|u\|+|u|_p, 
\end{align*}
where $|\cdot|_p$ and $\|\cdot\|$ denote the usual norms in $L^p(\rn)$ and $D^{1,2}(\rn)$, respectively. 

We say that $u \in X^p$ is a solution to \eqref{sublinear} if 
\begin{equation}\label{solution}
\int_{\R^N} \nabla u \nabla \varphi = \int_{\R^N} Q_\Omega \, f_p(u)\varphi \quad \text{ for all } \varphi \in \mathcal{C}_c^\infty(\R^N),
\end{equation}
with $f_p$ as in \eqref{fp}.

 Let $I_p:X^p\to \r$ be given by
\begin{align}\label{Ipdef}
    I_p(u):=\frac{1}{2}\|u\|^2-\frac{1}{p}\int_{\rn}Q_\Omega |u|^p.
\end{align}

Notice that $I_p$ is of class $\cC^1$ if $p>1$ and $u$ is a critical point for $I_p$ if 
\[ I_p'(u) \varphi= \int_{\R^N} \nabla u \nabla \varphi - \int_{\R^N} Q_\Omega |u|^{p-2} u \varphi=0 \quad \text{ for any } \varphi \in X^p. \]
Therefore, the critical points of $I_p$, with $p >1$, are the  solutions of \eqref{sublinear}.

If $p=1$ the energy functional $I_1$ is not of class $\cC^1$. To define the notion of critical points and Palais-Smale sequences we use some concepts from non-smooth critical point theory. 

Let us first recall some definitions and properties from \cite{chang}, see also \cite{ClarkeBook}. 
Let $X$ be a Banach space, and let $F: X \to \R$ be  locally Lipschitz. Its generalized gradient  at a point $x$ is the set
\[ \partial F(x):=\left \{ g \in X^*: \, \langle g, v\rangle  \le \limsup_{y \to x, \lambda \to 0^+}\frac{F(y + \lambda v) - F(y)}{\lambda} \quad \text{ for all } v \in X \right\}.   \]
Note that
\begin{equation}\label{C1}
\partial F(x)=\{F'(x)\}\quad\text{ if }F\text{ is of class }\cC^1.
\end{equation}
Furthermore,
\begin{equation}\label{sum} 
\partial(F_1+F_2)(x) \subseteq\partial(F_1)(x) + \partial(F_2)(x),
\end{equation}
\begin{equation}\label{hom} 
\partial (\lambda F)(x)=\lambda \partial F(x) \quad \text{ for any } \lambda \in \R,
\end{equation} 
and
$\partial F(x)$ coincides with the subdifferential of $F$ at $x$ in the sense of \cite{Ekeland} if $F$ is convex; see \cite[Section 1, Propositions (3)-(5)]{chang}. 

Using these facts, we now prove that
\begin{equation}\label{gradient I 1}
\partial I_1(u) \subseteq \left\{ g: X^1 \to \R:g(\varphi)=\int_{\R^N} \nabla u \nabla \varphi - \int_{\R^N} Q_\Omega s \varphi\text{ for all }\vp\in X^1 \text{ and some } s \in S(u) \right\},
\end{equation}
where
\begin{align*}
S(u):=\left\{s:\rn\to\r\::\: s(x)=1 \text{ if }u(x)>0; s(x)=-1\text{ if } u(x)<0; s(x)\in[-1,1]\text{ if } u(x)=0 
\right\}.
\end{align*}

Indeed, by \eqref{sum} and \eqref{hom}, 
\[ \partial I_1(u) \subseteq \frac 12 \partial (\norm{u}^2) - \partial \left(\int_{\R^N} Q_\Omega |u|\right) \]
and, by \eqref{C1},
\[\partial (\norm{u}^2) = \left \{2\langle u,\cdot\rangle: X^1 \to \R\text{ \ given by \ }2\langle u,\vp\rangle=2\int_{\R^N} \nabla u \nabla \varphi \right \}.  \]
As $u\mapsto \int_{\Omega} |u|$ is convex, its generalized gradient is its subdifferential, that is, for any $u \in X^1$,
\begin{align*}
\partial\left( \int_\Omega |u| \right) = \bigcup_{s\in S(u)}\left \{ g_{\o,s} : X^1 \to \R\text{ \ given by \ } g_{\o,s}(\varphi)=\int_{\Omega} s \varphi\right \}.
\end{align*}
Similarly,
\begin{align*}
\partial\left( \int_{\R^N \setminus \Omega} |u| \right) = \bigcup_{s\in S(u)}\left \{ g_{\R^N \smallsetminus \Omega,s} : X^1 \to \R\text{ \ given by \ } g_{\R^N \smallsetminus \Omega,s}(\varphi)=\int_{\R^N \smallsetminus\Omega} s \varphi\right \}.
\end{align*}
Thus,
\begin{align*} \partial \left(\int_{\R^N} Q_\Omega |u|\right) &\subseteq \partial\left( \int_\Omega |u| \right) - \partial\left( \int_{\R^N \setminus \Omega} |u| \right)
= \bigcup_{s\in S(u)}\left\{ g_s: X^1 \to \R\text{ \ given by \ } g_s(\varphi)=\int_{\R^N} Q_\Omega s \varphi \right\},
\end{align*}
which shows \eqref{gradient I 1}.

Following \cite[Definition 2.1]{chang} we say that
\begin{definition}\label{crit point}
$u \in X^1$ is a critical point of $I_1$ if $0 \in \partial I_1(u)$.
\end{definition}

Then, the following holds.

\begin{lemma}\label{opt sol}
If $u \in X^1$ is a critical point of $I_1$, then $u$ is a solution to \eqref{sublinear}.
\end{lemma}
\begin{proof}  
Let $0 \in \partial I_1(u)$, then, by \eqref{gradient I 1}, there exists $s \in S(u)$ such that 
\begin{equation}\label{eq s} \int_{\R^N} \nabla u \nabla \varphi= \int_{\rn} Q_\Omega s \varphi, \quad \text{ for all } \varphi \in \mathcal{C}_c^\infty(\R^N).  \end{equation}
Notice that $|Q_\Omega s| \le 1$, thus we can apply \cite[Lemma B.3]{S00} to conclude that $u \in L^q_{loc}(\R^N)$ for any $q < \infty$. 
Now, elliptic regularity \cite[Theorem B.2]{S00} yields $u \in W^{2, q}_{loc}(\R^N)$ for any $1 < q < \infty$. 
Therefore, we can apply \cite[Lemma 7.7]{GT}
to conclude that in the set $\{ u=0 \}$ one has $\nabla u =0$ a.e. A second application of \cite[Lemma 7.7]{GT} gives 
\begin{align}\label{delta}
\Delta u=0\qquad \text{ a.e. in $\{ u=0 \}$.}
\end{align}
By \eqref{eq s} one has that $-\Delta u= Q_\Omega s$ a.e. in $\rn$; hence, by \eqref{delta}, $s=0$ a.e. in the set $\{u =0\}$, which in turn proves that $s= \sgn(u)$ a.e. in $\rn$, namely $u$ satisfies \eqref{solution} with $p=1$, i.e., $u$ is solution to \eqref{sublinear}. 
\end{proof}

\section{Ground states and properties of the solutions}\label{sec:gs}

\subsection{Existence of ground states}

Let $p \in [1, 2)$ and set  
$$\mu_p:=\inf_{u\in X^p} I_p(u).$$ 
We first show that $\mu_p$ is attained. 

\begin{lemma} \label{lem:minimizer}
\begin{itemize}
\item[$(i)$] $-\infty<\mu_p<0$.
\item[$(ii)$] $\mu_p$ is attained at a nonnegative $w\in X^p$, and any minimizer of $I_p$ is a solution to \eqref{sublinear}.
\end{itemize}
\end{lemma}

\begin{proof}
$(i):$ \underline{Case $p>1$.} Using the Hölder and the Sobolev inequalities we see that
\begin{equation}\label{eq:basic inequality}
\io|u|^p\leq |\o|^{(2^*-p)/2^*}\Big(\irn |u|^{2^*}\Big)^{p/2^*}\leq C\|u\|^p\qquad\text{for every \ } u\in D^{1,2}(\rn).
\end{equation}
Therefore,
\begin{equation}\label{eq:bounded}
I_p(u)\geq \frac{1}{2}\|u\|^2-\frac{1}{p}\io|u|^p\geq\frac{1}{2}\|u\|^2-C\|u\|^p\qquad\text{for every \ } u\in X^p
\end{equation}
and, as $p\in(1,2)$, this implies that $\mu_p>-\infty$.

Let $\o_0$ be a connected component of $\o$ and let $u_0$ be the positive solution to the problem 
\begin{align}\label{o0}
-\Delta u = |u|^{p-2}u, \qquad u\in D^{1,2}_0(\o_0),    
\end{align}
where, as usual, $D^{1,2}_0(\o_0)$ denotes the closure of $\cC^\infty_c(\o_0)$ in $D^{1,2}(\rn)$. Then,
$$\mu_p \leq I_p(u_0)=\frac{1}{2}\|u_0\|^2-\frac{1}{p}\int_{\o_0}|u_0|^p=\frac{p-2}{2p}\|u_0\|^2<0;$$
see, for example, \cite{BFMST18}. 

\underline{Case $p=1$.} The proof follows the same argument. In this case a nonnegative solution of \eqref{o0} is given by the torsion function of $\o_0$, namely, the solution of 
\begin{align*}
-\Delta u = 1, \qquad u\in D^{1,2}_0(\o_0).
\end{align*}

$(ii):$ \underline{Case $p>1$.} Let $(u_k)$ be a sequence in $X^p$ such that $I_p(u_k)\to\mu_p$. It follows from \eqref{eq:bounded} that $(u_k)$ is bounded in $D^{1,2}(\rn)$. So, after passing to a subsequence, $u_k\rh u$ weakly in $D^{1,2}(\rn)$, $u_k\to u$ in $L^p_{loc}(\rn)$ and a.e. in $\rn$. By Fatou's lemma, $u\in L^p(\rn)$ and
\begin{align*}
\mu_p &\leq I_p(u)=\frac{1}{2}\|u\|^2+\frac{1}{p}\int_{\rn\smallsetminus\o}|u|^p-\frac{1}{p}\io|u|^p \\
&\leq \liminf_{k\to\infty}\frac{1}{2}\|u_k\|^2+\liminf_{k\to\infty}\frac{1}{p}\int_{\rn\smallsetminus\o}|u_k|^p-\lim_{k\to\infty}\frac{1}{p}\io|u_k|^p \\
&\leq\lim_{k\to\infty}I_p(u_k)=\mu_p.
\end{align*}
Therefore, $I_p(u)=\mu_p$. Since $I_p(|u|)=I_p(u)$, $w:=|u|$ is a nonnegative minimizer of $I_p$.
A minimizer for $I_p$ is a critical point of $I_p$, hence, it is a solution to \eqref{sublinear}.

\underline{Case $p=1$.} We repeat the same argument to prove the existence of a nonnegative minimizer $w$ of $I_1$. By \cite[Proposition 6]{Clarke}, any minimizer of $I_1$ is a critical point of $I_1$ in the sense of Definition~\ref{crit point}. Thus, Lemma~\ref{opt sol} implies that it is a solution to \eqref{sublinear}. 
\end{proof}

\subsection{Properties of the solutions}

\begin{lemma}\label{Linfty:lem}
Let $u\in X^p$ be a solution of \eqref{sublinear}. Then:
\begin{itemize}
    \item[(i)] If $p >1$,  $u\in \cC^{1,\alpha}_{loc}(\rn)\cap \cC^{2,\alpha}_{loc}(\Omega)\cap \cC^{2,\alpha}_{loc}(\rn\backslash\Omega)$ for some $\alpha\in(0,1)$;
    \item[(ii)]  If $p=1$, $u\in \cC^{1,\alpha}_{loc}(\rn)$  for some $\alpha\in(0,1)$. 
\end{itemize}
Moreover, for any $p \in [1, 2)$, 
there is $C=C(\Omega,p)>0$ such that
\begin{align}\label{localbd}
\|u\|_{\cC^{1,\alpha}(\Omega)}<C.
\end{align}
\end{lemma}
\begin{proof}
Since $u\in X^p$ is a solution of \eqref{sublinear}, then 
\begin{align*}
-\infty<\mu_p \leq I_p(u)=\frac{1}{2}\int_{\rn}|\nabla u|^2-\frac{1}{p}\int_{\rn}Q_\Omega|u|^p
=\left(\frac{1}{2}-\frac{1}{p}\right)\|u\|^2\leq0,
\end{align*}
that is, 
\begin{align*}
\|u\|^2 \in \left[0,\mu_p \left(\frac{2p}{p-2}\right)\right].
\end{align*}
Notice that $|\mu_p|$ is bounded by Lemma~\ref{lem:minimizer}, thus a standard Moser iteration argument, see \cite[B.3 Lemma]{S00}, shows that $u\in L^q_{loc}(\rn)$ for any $q >1$. Then, local regularity results \cite[Theorem B.2]{S00} yield that $u\in W^{2,q}_{loc}(\rn)$ for all $q>0$, which implies that $u\in \cC^{1,\alpha}_{loc}(\rn)$. The fact that $u\in \cC^{2,\alpha}_{loc}(\Omega)\cap \cC^{2,\alpha}_{loc}(\rn\backslash\Omega)$ if $p>1$ follows from Schauder estimates, see \cite[Appendix B]{S00}. The bound \eqref{localbd} is also a consequence of these arguments.
\end{proof}

Set $g_+:=\max\{g,0\}$.
\begin{lemma}\label{compact:sup:lem}
If $u\in X^p$ is a solution of \eqref{sublinear}, then there is $t_0=t_0(\Omega,p)>0$ so that 
\begin{align*}
|u(x)|\leq \left(\frac{2p}{(2-p)^2}\right)^{\frac{1}{p-2}}(t_0-|x|)_+^{\frac{2}{2-p}}\qquad \text{ for all }x\in \rn.
\end{align*}
In particular, $u$ has compact support and $u\in L^\infty(\rn)$.
\end{lemma}
\begin{proof}
\underline{Case $p>1$.}
By Lemma~\ref{Linfty:lem}, $u$ is locally bounded, and then, by \eqref{localbd}, there is $t_0=t_0(\Omega,p)>0$ such that
\begin{align}\label{t0}
|u(x)|<\left(\frac{2p}{(2-p)^2}\right)^{\frac{1}{p-2}}(t_0-|x|)_+^{\frac{2}{2-p}}\qquad \text{ for all }x\in \overline{\Omega}.
\end{align}
Let $h(|x|):=\left(\frac{2p}{(2-p)^2}\right)^{\frac{1}{p-2}}(t_0-|x|)_+^{\frac{2}{2-p}}$. A direct computation shows that 
$h''= h^{p-1}$ in $(0,\infty)$. Since $h$ is nonincreasing,
\begin{align*}
-\Delta h(|x|)
=-h''(|x|)-\frac{N-1}{|x|}h'(|x|)\geq -h^{p-1}(|x|)\qquad \text{ for }|x|>0.
\end{align*}
Let $w(x):=h(|x|)-u(x)$ and $S:=\{x\in \rn\backslash \Omega\::\: u(x)>h(|x|)\}.$ By \eqref{t0}, we have that $\partial S\cap  \partial\Omega=\emptyset$. 

Assume, by contradiction, that $S\neq \emptyset$. Then, since $u \in \cC^{2, \alpha}_{loc}(\R^N \setminus \Omega)$ by Lemma~\ref{Linfty:lem}, 
\begin{align*}
-\Delta w(x)\geq |u(x)|^{p-2}u(x)-h^{p-1}(|x|)\geq 0 \quad \text{ in }S,\qquad w=0\quad \text{ on }\partial S.
\end{align*}
By applying the weak maximum principle \cite[Theorem 8.1]{GT}, 
we would have that $w\geq 0$ in $S$, but this contradicts the definition of $S.$
Therefore, $S=\emptyset$ and $u(x)<h(x)$ in $\rn$. Repeating the argument for $-u$ instead of $u$ yields that $-u(x)<h(x)$ in $\rn$ and 
the claim follows.

\underline{Case $p=1$.}
As above, one can show that there exists $t_0=t_0(\Omega) >0$ such that \eqref{t0} holds, namely 
\[ |u(x)| < \frac 12 (t_0 - |x|)_+^2 \quad \text{ for all }x \in \overline \Omega. \]
As in the case $p>1$ we define $h(|x|):= \frac 12 (t_0 - |x|)_+^2 $. It satisfies $h''=1$ in $(0, \infty)$ and $-\Delta h(|x|) \ge -1$ for $|x|>0$. 
Let $w(x)$ and $S$ be defined as for $p>1$. Here, since $u$ does  not have $\cC^{2, \alpha}$ regularity, we need to use the weak formulation. 
Let $\varphi \in \mathcal{C}_c^\infty(\R^N)$ with $\varphi \ge 0$ and $\text{supp}(\varphi) \subset \subset S$. Notice that this is possible because $S$ is an open set since $\partial S \cap \partial \Omega = \emptyset$. Then 
\[ \int_{\R^N} \nabla w \nabla \varphi =-\int_{\R^N}\Delta h \varphi - \int_{\R^N} Q_\Omega \sgn(u) \varphi \ge -\int_{\R^N} \varphi + \int_{\R^N} \sgn(u) \varphi =0, \]
as $u(x) > h(|x|) \ge 0$ in $S$. Moreover, $w=0$ on $\partial S$. By the weak maximum principle \cite[Theorem 8.1]{GT}, 
we would have $w\geq 0$ in $S$, contradicting the definition of $S.$ 
Therefore, $S=\emptyset$ and $u(x)<h(x)$ in $\rn$. Repeating the argument for $-u$ instead of $u$ yields that $-u(x)<h(x)$ in $\rn$ and 
the claim follows. 
\end{proof}

\begin{remark}\label{scaling}
    Notice that a simple scaling argument shows that if $u$ is a solution to \eqref{sublinear}, then $v_r(x):=r^{\frac{2}{2-p}} u(x/r)$ is a solution to 
    \[ -\Delta v_r= Q_{\Omega_r} |v_r|^{p-2}v_r \quad \text{ in } \R^N,  \]
    where $\Omega_r:=r \Omega$.
    In particular 
    \begin{equation*}
        |v_r|_\infty \to \infty \, \text{ as } \, r \to \infty \quad \text{ and }  \quad |v_r|_\infty \to 0 \, \text{ as } \, r \to 0. 
    \end{equation*}
\end{remark}
% \begin{remark}\label{restriction supp}
%     If $p=1$, and $u$ is a nonnegative solution to \eqref{sublinear}, then, choosing $\varphi \in \cC_c^\infty(\rn)$ such that $\varphi=1$ in $\supp(u)$, one has 
%     \[ 0 = \int_{\rn} \nabla u \nabla \varphi= \int_{\R^N} Q_\Omega \sgn(u) \varphi = |\Omega|- |\supp (u) \setminus \Omega|= 2 |\Omega|- |\supp(u)|,\]
%     whence
%     \[ 2 |\Omega|=|\supp(u)|. \]
%     In particular, this shows that, if $p=1$ and $\Omega$ is an ellipse with sufficiently large eccentricity or a thin enough annulus, then the support of any nonnegative solution to \eqref{sublinear} cannot be a ball. 
% \end{remark}

\begin{lemma}\label{pos:lem}
Let $u$ be a nontrivial nonnegative solution of \eqref{sublinear}.
\begin{itemize}
\item[$(i)$] For each connected component $\omega$ of $\o$ we have that, either $u\equiv 0$ on $\overline{\omega}$, or $u> 0$ on $\overline{\omega}$.
\item[$(ii)$] $u> 0$ in at least one connected component $\omega$ of $\o$.
\end{itemize}
\end{lemma}

\begin{proof}
Since $u\not\equiv 0$ in $\rn,$
\begin{align*}
0<|\nabla u|^2_2 = \int_{\rn}Q_\Omega f_p(u) u\, dx \le \int_{\Omega} f_p(u) u \, dx,
\end{align*}
and therefore $u\not\equiv 0$ in $\Omega$ and, as a consequence, $u\not\equiv 0$ in some connected component $\omega$ of $\Omega$.

If $u\not\equiv 0$ in the connected component $\omega$ of $\Omega$, then, for any $\varphi \in \mathcal{C}_c^\infty(\omega)$ such that  $\varphi \ge 0$, 
\begin{align}\label{max princ}
\int_{\R^N} \nabla u \nabla \varphi =\int_{\R^N} f_p(u) \varphi \geq 0. 
\end{align}
Since $\omega$ is connected and $u\in C^{1,\alpha}_{loc}(\rn)$ (as shown in Lemma~\ref{Linfty:lem}), the strong maximum principle yields $u>0$ in $\omega$ and Hopf's lemma implies that $u>0$ in $\overline{\omega}$; see, for example, \cite[Theorem 1.28]{DP}.
\end{proof}

\begin{lemma} \label{lem:sign of minimizer}
If $\o$ is connected, $v\in X^p$ and $I_p(v)=\mu_p$, then, either $v\geq 0$, or $v\leq 0$.
\end{lemma}

\begin{proof}
\underline{Case $p>1$.} 
Assume, by contradiction, that $I_p(v)=\mu_p$ and that $v$ changes sign. Write $v=v^++v^-$ with $v^+:=\max\{v,0\}$ and $v^-:=\min\{v,0\}$. Then, 
\begin{equation} \label{eq:v pm}
    0=I_p'(v)v^\pm=I_p'(v^\pm)v^\pm=\|v^\pm\|^2-\irn Q_\o|v^\pm|^p.
\end{equation}
Set $u:=|v|$. Then, $I_p(u)=\mu_p$ and, by Lemma~\ref{pos:lem}, $u>0$ on $\overline{\o}$. It follows that, either $\{x\in\rn:v(x)>0\}\subset\rn\smallsetminus\o$, or $\{x\in\rn:v(x)<0\}\subset\rn\smallsetminus\o$. Assume the former holds true. Then we derive from \eqref{eq:v pm} that
$$\|v ^+\|^2=\irn Q_\o|v^+|^p=-\irn|v^+|^p\leq0,$$
which implies that $v^+=0$. This is a contradiction. The other case is similar.

\underline{Case $p=1$.} If $v$ is a minimizer for $I_1$, then, by Lemma~\ref{lem:minimizer}, it is a solution to \eqref{sublinear}, namely  
\[ 0 = \int_{\R^N} \nabla v \nabla \varphi - \int_{\rn} Q_\Omega \, \sgn(v) \varphi \quad \text{ for all } \varphi \in \mathcal{C}_c^\infty(\R^N). \]
Thus
\[ 0 = \norm{v^\pm}^2 - \int_{\R^N} Q_\Omega |v^\pm|. \]
From here, the proof follows as in the case $p>1$. 
\end{proof}

\begin{remark}\label{rem:sign-changing}
Note that the previous lemma does not exclude the existence of a sign-changing minimizer if $\o$ is not connected.
\end{remark}

Let us show that $I_p$ has a unique nonnegative minimizer. For this, we need the following auxiliary result. 

\begin{lemma}\label{lem:2}
Let $p \in [1, 2)$.
For $u,v\in X^p$, $u\geq 0$, $v\geq 0$ and $t\in [0,1],$ let
\begin{align}\label{gamma}
\gamma(t):=((1-t)u^p+tv^p)^\frac{1}{p}.
\end{align}
If $|\nabla u|\neq |\nabla v|$, then the function 
\begin{align*}
t\mapsto I_p(\gamma(t))
\end{align*}
is strictly convex for $t\in[0,1]$.
\end{lemma}

\begin{proof}
It is known that 
\begin{align*}
|\nabla \gamma(t)|\leq ((1-t)|\nabla u|^p+t|\nabla v|^p)^\frac{1}{p}\qquad \text{ for }t\in [0,1].
\end{align*}
Since $p<2$, it follows that 
\begin{align*}
|\nabla \gamma(t)|^2\leq (1-t)|\nabla u|^2+t|\nabla v|^2\qquad \text{ for }t\in [0,1],
\end{align*}
and the inequality is strict whenever $|\nabla u|$ and $|\nabla v|$ are different and $t\in (0,1)$. This inequality implies that 
\begin{align*}
    t\mapsto \|\gamma (t)\|^2
\end{align*}
is strictly convex, see the proof of \cite[Lemma 3.5]{BFMST18}. The proof is now finished, because 
\begin{align*}
\int_{\rn}Q_\Omega |\gamma(t)|^p
=\int_{\rn}Q_\Omega ((1-t)u^p+tv^p)=(1-t)\irn Q_\o u^p+t\irn Q_\o v^p
\end{align*}
for all $t\in[0,1]$.
\end{proof}

\begin{theorem}\label{uniquenessOmega}
$I_p$ has a unique nonnegative minimizer $w_p\in X^p$. Furthermore, if $\Omega$ is connected, then $w_p$ is the only nonnegative solution of \eqref{sublinear}. 
\end{theorem}

\begin{proof}
The first statement can be proved in the same way for the case $p>1$ and $p=1$. 
    By Lemma~\ref{lem:minimizer}, $I_p$ has a nonnegative minimizer. Assume that $u$ and $v$ are two nonnegative minimizers of $I_p$, and let $\gamma$ be given by \eqref{gamma}. Then, by Lemma~\ref{lem:2}, the function $h(t):=I_p(\gamma(t))$ is strictly convex in $[0,1]$ if $|\nabla u|\neq |\nabla v|$. Recalling that $u$ and $v$ have compact support, $h(t)$ is strictly convex if $u \ne v$.  
    Since $h(t)\geq\mu_p$ for all $t\in[0,1]$ and $h(0)=I_p(u)=\mu_p=I_p(v)=h(1)$, it follows that $u=v$. 

    To prove the second statement, we need to distinguish the case $p>1$ and the case $p=1$. We assume that $\Omega$ is connected. 

    \underline{Case $p>1$.}
    We use a standard argument based on the generalized Picone identity, see for example \cite[Proposition 2.3]{BR25}. Let $u,v$ denote two nonnegative solutions of \eqref{sublinear}. By Lemma~\ref{pos:lem}, $u>0$ and $v>0$ in $\Omega.$ Without loss of generality, we may assume that $v$ is the unique nonnegative minimizer of \eqref{sublinear}.  Let $\varepsilon > 0$ and let $\phi:= \frac{v^{p}}{(u+\varepsilon)^{p-1}}$. Since $v\in L^\infty(\rn)$, we have that $\phi \in X^{p}$. By the generalized Picone identity \cite[Proposition 2.9]{BF14},
\[
\nabla u\cdot \nabla \left(\frac{v^{p}}{(u+\varepsilon)^{p-1}}\right)
\leq |\nabla u|^{2-p}|\nabla v|^{p} \quad \text{in } \rn.
\]
Hence, testing \eqref{sublinear} with $\phi$ and using Hölder's inequality,
\[
\int_{\rn} Q_\Omega v^{p}\left(\frac{u}{u+\varepsilon}\right)^{p-1} \, dx 
= \int_{\rn}\nabla u \cdot\nabla \left(\frac{v^{p}}{(u+\varepsilon)^{p-1}}\right) dx\leq 
\int_{\R^N} |\nabla u|^{2-p}|\nabla v|^{p}\, dx\leq |\nabla u|_2^{2-p}|\nabla v|_2^{p}.
\]

Recall that $u>0$ in $\Omega$. Then, by dominated convergence, we can let $\varepsilon \to 0$ and we have that 
\begin{equation*}
|\nabla v|^2_2
=\int_{\rn} Q_\Omega v^{p} \, dx 
\leq \int_{\rn\cap \{u>0\}} Q_\Omega v^{p} \, dx 
=\lim_{\eps\to 0}\int_{\rn} Q_\Omega v^{p}\left(\frac{u}{u+\varepsilon}\right)^{p-1} \, dx 
\leq |\nabla u|_2^{2-p}|\nabla v|_2^{p}.
\end{equation*}
This yields that $|\nabla v|_2\leq |\nabla u|_2$ and therefore
\begin{equation}\label{eq2.6}
I_p(v)\leq I_p(u) = -\frac{2-p}{2p} |\nabla u|_2^{2}
\leq -\frac{2-p}{2p} |\nabla v|_2^{2}
= I_p(v).
\end{equation}
Hence $u$ is also a minimizer of $I_p$ and the first statement implies that $u \equiv v$, as claimed.

\underline{Case $p=1$.} Let $u, v$ be two nonnegative solutions of \eqref{sublinear}, and assume that $v$ is the nonnegative minimizer. 
    By Lemma~\ref{pos:lem} we know that $u>0$ and $v>0$ in $\overline \Omega$. We now have
    \[ \int_{\R^N} Q_\Omega \sgn(u) v = \int_{\R^N} \nabla u \nabla v \le |\nabla u|_2 |\nabla v|_2. \]
    Observe that 
    \[ \Omega \subset \text{supp}(v) \cap \text{supp}(u). \]
    Hence
    \[ |\nabla v|_2^2 = \int_{\R^N} Q_\Omega v = \int_{\text{supp}(v)} Q_\Omega v \le \int_{\text{supp}(v) \cap \text{supp}(u)} Q_\Omega  \sgn (u) v \le |\nabla u|_2 |\nabla v|_2. \]
    Therefore, $|\nabla v|_2 \le |\nabla u|_2$. 
    We now conclude
    \[ I_1(v) \le I_1(u) = -\frac 12 |\nabla u|_2^2 \le -\frac 12 |\nabla v|_2^2 = I_1(v), \]
    whence $u$ is also a least energy solution, thus $u \equiv v$. 
\end{proof}

\begin{lemma}\label{pos:ground state}
    Let $u$ be the nonnegative minimizer for $I_p$. Then $u>0 $ on $\overline \Omega$.
\end{lemma}
\begin{proof}
    If $\Omega$ is connected, this immediately follows from Lemma~\ref{pos:lem}. 
    
    Let us assume that $\Omega$ has two connected components $\omega_1$ and $\omega_2$ (the general case being completely analogous).
    Assume without loss of generality that $u$ is non trivial in $\omega_1$, thus, by Lemma~\ref{pos:lem}, one has $u >0$ on $\overline \omega_1$. 
Again by the maximum principle, $u>0$ or $u \equiv 0$ in $\overline \omega_2$. Assume by contradiction that $u \equiv 0$ in $\omega_2$. Let us take any nonnegative $\psi \in C_c^\infty(\omega_2)$, $\lambda >0$, and consider $\varphi_\lambda:= \lambda \psi$. Thus
\[ \int_{\omega_2} |\nabla \varphi_\lambda|^2 = \lambda^2 \int_{\omega_2} |\nabla \psi|^2, \quad \int_{\omega_2} |\varphi_\lambda|^p = \lambda^p \int_{\omega_2} |\psi|^p. \]
Since $p \in [ 1, 2)$, we can choose $\lambda>0$ small enough such that 
\[ \lambda^2 \int_{\omega_2} |\nabla \psi|^2 - \frac 2p \lambda^p \int_{\omega_2} |\psi|^p \le 0. \]
Define $u_\lambda :=u+ \varphi_\lambda$. Then
\begin{align*} I_p(u_\lambda) & = \frac 12 \int_{\R^N} |\nabla u_\lambda|^2 - \frac 1p \int_\Omega |u_\lambda|^p + \frac 1p \int_{\R^N \setminus \Omega} |u_\lambda|^p \\
&= \frac 12 \int_{\R^N} |\nabla u|^2 + \frac 12 \int_{\omega_2} |\nabla \varphi_\lambda|^2 - \frac 1p \int_{\omega_1} |u|^p - \frac 1p \int_{\omega_2} |\varphi_\lambda|^p +\frac 1p \int_{\R^N \setminus \Omega} |u|^p \\  
&\le I_p(u) = \mu_p.
\end{align*}
Thus $u_\lambda$ is a nonnegative minimizer for $I_p$ and, by uniqueness (see Theorem~\ref{uniquenessOmega}), it must coincide with $u$, a contradiction. 
\end{proof}

\subsection{Starshapedness and Lipschitz regularity}

In this section, we show that if $w\in X^p$ is the nonnegative ground state of \eqref{sublinear} and $\Omega$ is a smooth bounded starshaped domain, then the support of $w$ is also starshaped. Furthermore, if $\Omega$ is strictly starshaped, then the boundary of the support of $w$ is locally the graph of a Lipschitz function. 

To show this, we follow \cite{spruck}, where the author studies the following problem 
\begin{align}\label{sublinear2}
    -\Delta v = f(x)v^\beta\quad \text{ in }\rn,\qquad v\geq 0,
\end{align}
where $\beta\in (0,1)$ and $f$ is a locally Lipschitz function satisfying that $x\cdot \nabla f\leq 0.$  Problem \eqref{sublinear2} is closely related to \eqref{sublinear}, but our coefficient $Q_\Omega$ has a discontinuity on $\partial \Omega$, and this yields some complications in some of the arguments in \cite{spruck}.

\begin{remark}
The proofs of the results \cite[Lemma 3.2, Corollary 3.3, Corollary 3.4, and Lemma 3.5]{spruck} carry over to our setting word by word, by noting that their set $\{f<0\}$ is simply $\rn\backslash \Omega$ in our context. As a consequence, \cite[Corollary 3.6]{spruck} also holds for \eqref{sublinear} with $p\in(1,2)$, and we obtain that, if $w$ is the nonnegative ground state of \eqref{sublinear}, then there is $C>0$ such that
\[
w(x)\leq \dist(x,\mathbb R\backslash K)^\frac{2}{2-p}\qquad \text{ for all }x\in \rn.
\]
We do not need this estimate in this paper, but we mention it as a remark of independent interest. 
\end{remark}

The following is the analog of \cite[Lemma 3.8]{spruck} for \eqref{sublinear}. The proof follows a Bernstein method to obtain a priori gradient bounds (as in \cite{spruck}). 

\begin{lemma}\label{lem:bdsSpruck}
Let $w$ be the nonnegative ground state of \eqref{sublinear}, then there is $C>0$ such that
\begin{align*}
    |\nabla w|^2\leq Cw^{p}\qquad \text{for all $x\in K:=\supp(w)$.}
\end{align*}
\end{lemma}
\begin{proof}
For $\varepsilon > 0$, let $\xi_\eps := \frac{|\nabla w|^{2}}{(w+\varepsilon)^{p}}.$  Note that $\xi_\eps \ge 0$ in $\rn$ and $\xi_\eps=0$ on $\partial K$. Suppose $\xi_\eps$ achieves its maximum at some $x_\eps$.  If $x_\eps\in  \overline{\Omega}$ for all $\eps>0,$ then $x_\eps\to x_0$ for some $x_0\in \overline{\Omega}$, $w(x_0)>0,$ and $\xi_\eps(x_\eps)\to \frac{|\nabla w|^{2}}{w^{p}}(x_0)=:C$ as $\eps\to 0;$ namely, 
\begin{align*}
\frac{|\nabla w(x)|^{2}}{(w(x)+\eps)^{p}}\leq \max_{\rn}
\frac{|\nabla w|^{2}}{(w+\eps)^{p}}
=\xi_\eps(x_\eps)=C+o(1)\qquad 
\text{for any $x\in \rn$}.
\end{align*}
Since $C>0$ is independent of $\eps$, the claim would follow. 

Hence, passing to a subsequence, we may assume that $x_\eps\in K\backslash \overline{\Omega}$ for all $ \eps>0$. Observe that
\begin{align}\label{nablaeq}
    \nabla \xi_\eps 
    =\nabla \left( \frac{|\nabla w|^{2}}{(w+\varepsilon)^{p}} \right)
    =\frac{\nabla (|\nabla w|^{2})}{(w+\varepsilon)^{p} }
- p \frac{|\nabla w|^{2} \nabla w}{(w+\eps)^{p+1}}. 
\end{align}
We use $K^\circ$ to denote the interior of $K$.  For every open set $D\subset\subset K^\circ \backslash \overline{\Omega},$ there is $c_D>0$ such that $w>c_D$ in $D;$ in particular, $w^{p-1}\in C^\infty(D)$. Then, using Schauder estimates we have that $w$ is smooth in $K^\circ \backslash \overline{\Omega}$. Moreover, 
\begin{align*}
\Delta \xi_\varepsilon
&= \operatorname{div}\left((w+\varepsilon)^{-p}\nabla(|\nabla w|^{2})\right)
- p\operatorname{div}\left((w+\varepsilon)^{-(p+1)}|\nabla w|^{2}\nabla w\right) \\
&= (w+\varepsilon)^{-p}\Delta(|\nabla w|^{2})
- p (w+\varepsilon)^{-(p+1)} \nabla w\cdot\nabla(|\nabla w|^{2})
- p (w+\varepsilon)^{-(p+1)} \nabla(|\nabla w|^{2})\cdot\nabla w\\
&\quad - p (w+\varepsilon)^{-(p+1)} |\nabla w|^{2}\Delta w + p(p+1)(w+\varepsilon)^{-(p+2)} |\nabla w|^{4}
\qquad \text{ in }K^\circ \backslash \overline{\Omega}.
\end{align*}

Using the identity $\Delta(|\nabla w|^{2})
= 2|D^{2}w|^{2} + 2\nabla w\cdot\nabla(\Delta w)$ and \eqref{sublinear},
\begin{align*}
\Delta \xi_\varepsilon
&= 2(w+\varepsilon)^{-p}
\bigl(|D^{2}w|^{2} + (p-1)|\nabla w|^2w^{p-2}\bigr)- 2p (w+\varepsilon)^{-(p+1)} \nabla w\cdot\nabla(|\nabla w|^{2}) \\
&\quad
- p (w+\varepsilon)^{-(p+1)} |\nabla w|^{2} w^{p-1}
+ p(p+1)(w+\varepsilon)^{-(p+2)} |\nabla w|^{4}\qquad \text{ in }K^\circ \backslash \overline{\Omega}.
\end{align*}

Now, using that $\nabla \xi_\eps(x_\eps) =0$ and \eqref{nablaeq}, 
\begin{align}\label{atxeps}
\nabla w(x_\eps)\cdot\nabla(|\nabla w(x_\eps)|^{2})
= p \frac{|\nabla w(x_\eps)|^{4}}{w(x_\eps)+\varepsilon},    
\end{align}
and then
\begin{align}
0\geq \Delta \xi_\varepsilon(x_\eps)
&= \Bigg[(w+\varepsilon)^{-p}
\Bigl(2|D^{2}w|^{2} + 2(p-1) w^{p-2} |\nabla w|^{2}\Bigr) \notag\\
&\quad
- p (w+\varepsilon)^{-(p+1)} w^{p-1} |\nabla w|^{2}
+ p(1-p)(w+\varepsilon)^{-(p+2)} |\nabla w|^{4}\Bigg](x_\eps).\label{2}
\end{align}

Since $|\nabla w(x_\eps)|>0$ we may rotate coordinates so that
\begin{align*}
|\nabla w(x_\eps)| = \partial_1 w(x_\eps) > 0\qquad \text{ and }\qquad \partial_jw(x_\eps)=0\quad \text{ for $j=2,\ldots,N$.}     
\end{align*}
Then, by \eqref{atxeps},
\[
2 \partial_{11}w(x_\eps) \partial_1w(x_\eps)
= p \frac{(\partial_1 w)^{3}}{w+\varepsilon}(x_\eps),\qquad \text{i.e.,}\qquad 
\partial_{11}w(x_\eps)
= \frac{p}{2}\frac{|\nabla w|^{2}}{w+\varepsilon}(x_\eps).
\]
Hence, $|D^{2}w|^{2}(x_\eps)\ge (\partial_{11}w(x_\eps))^{2}
= \frac{p^{2}}{4}\frac{|\nabla w(x_\eps)|^{4}}{(w(x_\eps)+\varepsilon)^{2}}.$ Using this and \eqref{2},
\[
0 \ge
\left[\frac{p(2-p)}{2}(w+\varepsilon)^{-(p+2)}|\nabla w|^{4}
+ (w+\varepsilon)^{-(p+1)} w^{p-2}
\bigl(2(p-1)(w+\varepsilon)-pw\bigr)|\nabla w|^{2}\right](x_\eps).
\]

Using that $\xi_\eps = \frac{|\nabla w|^{2}}{(w+\varepsilon)^{p}}$ and rearranging some terms, we obtain that 
$\frac{p(2-p)}{2}\xi_\eps^2 \leq 
B_\eps(x_\eps) \xi_\eps(x_\eps),$ where
\begin{align*}
B_\eps:=-\frac{w^{p-2}\bigl(2(p-1)(w+\varepsilon)-pw\bigr)}
{(w+\varepsilon)^{p-1}}
=\frac{w^{p-2}\bigl((2-p)w-2(p-1)\varepsilon\bigr)}
{(w+\varepsilon)^{p-1}}\leq
\frac{w^{p-2}\bigl((2-p)w\bigr)}
{(w+\varepsilon)^{p-1}}\leq 2-p.
\end{align*}
This implies that $\xi_\varepsilon(x_\varepsilon)\le (2-p)\frac{2}{(2-p)p}=\frac{2}{p}$ for all $\eps>0$ and the claim follows letting $\eps\to 0$.
\end{proof}

\begin{theorem}\label{thm:starshaped}
Let $\o$ be a smooth bounded domain and let $w$ be the nonnegative ground state of \eqref{sublinear} with $p\in(1,2)$.
If $\o$ is starshaped with respect to the origin , then 
$$\t:=\{x\in\rn:w(x)>0\}$$
is starshaped with respect to the origin.
\end{theorem}

\begin{proof}
Let $\phi:=w^{1-\beta}$ with $\beta=p-1$. Write $x\in\rn$ as $x=r\vartheta$ with $r:=|x|$ and $\vartheta:=\frac{x}{|x|}$. Note that
$$\{0\}\cup\Big\{x\in\rn\smallsetminus\{0\}:\frac{\phi(x)}{|x|^2}> 0\Big\}=\t.$$
Therefore, it suffices to show that
$$\partial_r\Big(\frac{\phi(r\vartheta)}{r^2}\Big)\leq 0\qquad\text{for all \ }x=r\vartheta\in\t\smallsetminus\{0\}.$$
Now,
$$\partial_r\Big(\frac{\phi(r\vartheta)}{r^2}\Big)=r^{-3}\big(r\partial_r\phi(r\vartheta)-2\phi(r\vartheta)\big).$$
We show next that
$$\zeta(x):=r\partial_r\phi(r\vartheta)-2\phi(r\vartheta)=x\cdot\nabla\phi(x)-2\phi(x)\leq 0\qquad\text{for all \ }x=r\vartheta\in\t.$$
Since $w\in W^{2, q}_{loc}(\R^N)$ (for any $q\geq 1$, see the proof of Lemma~\ref{Linfty:lem}), we have that $\zeta\in W^{1, q}_{loc}(\Theta)$. By Lemma~\ref{lem:bdsSpruck}, 
\begin{align}\label{gradbd}
|\nabla\phi|^2=(1-\beta)^2 w^{-2\beta}|\nabla w|^2\leq Cw^{1-\beta}=C\phi\qquad \text{ in $\t$.}
\end{align}
 Therefore,
$$\zeta=0\qquad\text{on \ }\partial\t.$$ 

Furthermore, $|\nabla \phi|^2= (1-\beta)^2 \phi^{-\frac{2\beta}{1-\beta}}|\nabla w|^2$ and then  
\begin{align*}
|\nabla w|^2=\frac{1}{(1-\beta)^2}|\nabla \phi|^2 \phi^{\frac{2\beta}{1-\beta}}.
\end{align*}
Moreover, $\partial_{ii}\phi
=\partial_i((1-\beta)w^{-\beta}\partial_i w)
=(1-\beta) \left(
(-\beta)w^{-\beta-1}|\partial_iw|^2+w^{-\beta}\partial_{ii} w
\right),$ which yields
\begin{align}\label{Deltaphi}
-\Delta \phi 
&= (1-\beta) \left(
\beta w^{-\beta-1}|\nabla w|^2+w^{-\beta}(-\Delta w)
\right)= (1-\beta) \left(
\beta \phi^{\frac{-\beta-1}{1-\beta}}
\left(\frac{1}{(1-\beta)^2}|\nabla \phi|^2 \phi^{\frac{2\beta}{1-\beta}}\right)
+w^{-\beta}Q_\Omega w^\beta
\right)\nonumber\\
&=
\frac{\beta }{(1-\beta)}\frac{|\nabla \phi|^2}{\phi}
+(1-\beta)Q_\Omega\quad \text{ in }\Theta.
\end{align}

Noting that $\Delta\zeta=x\cdot\nabla(\Delta\phi)$ and using equation \eqref{Deltaphi}, a straightforward computation shows that
\begin{align*}
 -\Delta \zeta  + b\cdot\nabla \zeta + c\zeta&=(1-\beta)r \partial_r (Q_\Omega)\quad \text{ in \ }\t,\qquad 
 b:=-\frac{2\beta}{1-\beta} \frac{\nabla\phi}{\phi},\quad 
 c:=\frac{\beta}{1-\beta} \frac{|\nabla\phi|^2}{\phi^2} \geq 0,
\end{align*}
in the sense of distributions.  Recall that $\partial_r \chi_\Omega
= -\left(\nu_\Omega(x)\cdot \frac{x}{|x|}\right)\,\delta_{\partial\Omega}$ in distributional sense, where $\nu_\Omega$ is the outer normal vector on $\partial \Omega$ and $\delta_{\partial\Omega}$ is the Dirac distribution supported on $\partial \Omega$.  As $Q_\Omega=\chi_\Omega-\chi_{\rn\backslash \Omega},$ we get
\begin{align*}
r \partial_r (Q_\Omega)
=
r\partial_r \chi_\Omega - r\partial_r \chi_{\rn\backslash\Omega}
=
-\left(\nu_\Omega(x)\cdot x\right)\,\delta_{\partial\Omega}+\left(\nu_{\rn\backslash\Omega}(x)\cdot x\right)\,\delta_{\partial\Omega}
=-2\left(\nu_\Omega(x)\cdot x\right)\,\delta_{\partial\Omega}
\end{align*}
in distributional sense. Hence,
\begin{align}\label{weak-zeta}
\int_\Theta \nabla \zeta \nabla \varphi+(b\cdot\nabla \zeta) \varphi + c\zeta\varphi 
= -2(1-\beta)\int_{\partial\Omega}\left(\nu_\Omega(x)\cdot x\right)\varphi \d\sigma\leq 0\qquad 
\text{for all \ }\varphi\in \cC^\infty_c(\t), \ \varphi\geq 0,
\end{align}
where we used that $\nu_\Omega(x)\cdot x\geq 0$ for $x\in \partial \Omega$, because $\Omega$ is a starshaped domain.  Observe that $b$ and $c$ might not be bounded close to $\partial \Theta$. To use a maximum principle, we implement a localization argument. Let $(\varepsilon_j)$ be a sequence of regular values of $\phi$ such that $\varepsilon_j\to 0$ as $j\to \infty$ and set
\[
\Theta_j:=\{x\in\Theta:\phi(x)>\varepsilon_j\}.
\]
Up to a subsequence, $\Omega\subset \subset \Theta_j \subset \subset\Theta$ and $b,c\in
L^\infty(\Theta_j)$, because, by \eqref{gradbd}, $|b|\leq C_\beta\varepsilon_j^{-1/2}$ and $0\leq c\leq C_\beta\varepsilon_j^{-1}
$ in $\Theta_j.$ Let $R>0$ be such that $\Theta\subset B_R$, then 
\[
\zeta\leq|\zeta|
\leq R|\nabla\phi|+2\phi
\leq R\sqrt{C\varepsilon_j}+2\varepsilon_j=:M_j\qquad \text{on
$\partial\Theta_j=\{\phi=\varepsilon_j\}$}.
\]
By \eqref{weak-zeta},
\[
-\Delta(\zeta-M_j)+b\cdot\nabla(\zeta-M_j)
+c(\zeta-M_j)\leq0
\quad\text{weakly in }\Theta_j,
\]
because $c\geq0$. Since $\zeta-M_j\leq0$ on
$\partial\Theta_j$, the weak maximum principle
\cite[Theorem 8.1]{GT} yields that $\zeta\leq M_j$ in $\Theta_j$. Letting $j\to\infty$ we obtain that $\zeta\leq0$ in $\Theta$.
\end{proof}

\begin{corollary}\label{Cor:starshaped}
Under the assumptions of Theorem~\ref{thm:starshaped}, if $\Omega$ is strictly starshaped with respect to the origin, then $\partial \t$ is Lipschitz. 
\end{corollary}
\begin{proof}
Note that, since $\Omega$ is strictly starshaped with respect to the origin, then there is $\eps>0$ such that
\begin{align*}
\nu_\Omega(x)\cdot (x-x_0)>0\qquad 
\text{for every \ }x_0\in B_\eps(0)\text{ \ and \ }x\in \partial \Omega.
\end{align*}
Hence, we can repeat the arguments in Theorem~\ref{thm:starshaped} to guarantee that 
\begin{align*}
    \frac{d}{d|x-x_0|}\left( \frac{\phi}{|x-x_0|^2} \right) \leq 0 \qquad \text{ a.e. in } \t.
\end{align*}
This yields that $\partial\t$ satisfies the interior and the exterior cone condition and the claimed Lipschitz regularity follows. 
\end{proof}

\begin{remark}[Regularity of the free boundary]\label{altphil}
Let $p\in(1,2)$ and let $K$ denote the support of the nonnegative ground state $w$. By Theorem \ref{thm:main}, we know that $K$ is compact and that $r:=\dist(\partial K,\Omega)>0$, because $w>0$ in $\overline{\Omega}$.  Let $x_0\in \partial K$, then
\begin{align*}
\Delta w = w^{p-1}\quad \text{in } U^+ := \{w>0\}\cap B_{r}(x_0),\qquad
w \geq 0\quad \text{in } U:=B_{r}(x_0).
\end{align*}
This is the framework of (local minimizers of) the so-called Alt–Phillips problem (for $p\in(1,2)$) and of the obstacle problem (for $p=1$). We refer to the introduction of \cite{RR25} for a recent state of the art on the known results regarding the regularity of the free boundary $\Gamma=\partial\{w>0\}$.  In particular, it is known that
\[
\Gamma=\mathcal{R}(w)\cup\mathcal{S}(w),
\]
where $\mathcal{R}(w)$ is the relatively open set of regular points and $\mathcal{S}(w)$ is the closed singular set. The regular part of the free boundary is a $C^{\infty}$ hypersurface, whereas the singular set is known to be $(N-1)$-rectifiable.

\end{remark}

\subsection{Asymptotic behavior of nonnegative ground states as 
\texorpdfstring{$p \to 1^+$}{pto1}}
Let  $w_p \in X^p$ be such that 
\[ I_p(w_p)=\mu_p\qquad\text{and}\qquad w_p\geq 0. \]

Next, let us show an auxiliary lemma. 
\begin{lemma}\label{bounded Xp}
Let $p \in [1, 2)$, and let $\mathcal{I}^-:=\{ u \in X^p: I_p(u) \le 0\}.$
\begin{itemize}
\item[$(i)$] There exists a constant $C=C(\Omega, N, p)>0$ such that 
\[ \norm{u}_{X^p} \le C \norm{u} \quad \text{ for any } u \in \mathcal{I}^-. \]
\item[$(ii)$] $\mathcal{I}^- $ is bounded in $X^p$. 
\end{itemize}
\end{lemma}

\begin{proof}
   $(i):$ Let $u \in \mathcal{I}^-$. Then, using \eqref{eq:basic inequality} and the fact that $I_p(u) \le 0$, we get
    \begin{equation}\label{estimate 1} \int_{\R^N \setminus \Omega} |u|^p \le \frac p 2 \norm{u}^2 + \int_{\R^N \setminus \Omega} |u|^p \le \int_\Omega |u|^p \le C \norm{u}^p, \end{equation}
which yields
\[ \norm{u}_{X^p} = \norm{u} + |u|_p \le \tilde C \norm{u}\qquad\text{for all \ }u\in\cI^-, \]
as claimed.

$(ii):$ By \eqref{estimate 1}, 
    \begin{equation}\label{bound 1} \norm{u}^2 \le\norm{u}^2 + \int_{\R^N \setminus \Omega} |u|^p \leq C \norm{u}^p \end{equation}
     for any $u \in \mathcal{I}^-$. Hence, $\mathcal{I}^- $ is bounded in $X^p$.
\end{proof}

\begin{proposition}\label{prop conv p 1}
After passing to a subsequence, $w_p \to w_1$ in $\mathcal{C}^{1, \alpha}_{loc}(\R^N)$ and $\mu_p \to \mu_1$  as $p \to 1$.   
\end{proposition}
\begin{proof}
    By dominated convergence,
    \[ \mu_p \le I_p(w_1) = \frac 12 \int_{\R^N} |\nabla w_1|^2 - \frac 1p \int_{\R^N}  Q_\Omega |w_1|^p \to I_1(w_1) = \mu_1 \quad \text{ as } p \to 1^+, \]
whence $\lim_{p \to 1} \mu_p \le \mu_1$. 

Moreover, by H\"older and Sobolev inequalities, 
\[ \mu_p \ge \frac 12 \norm{w_p}^2 - C(\Omega) \norm{w_p}^p,  \]
where $C(\Omega)$ is a positive constant independent of $p$. The function
\[g(t)=\frac 12 t^2 - C(\Omega) t^p\]
is bounded below by 
\[ g(t) \ge \left(\frac 12 -\frac 1p \right) \left( \frac 1 {pC(\Omega)} \right)^{\frac 2{p-2}} \to -\frac 12 C(\Omega)^2. \]
Therefore, $|\mu_p|$ is bounded uniformly in $p$. 
On the other hand 
\[ \norm{w_p}^2 = \mu_p \left( \frac 12 -\frac 1p \right)^{-1}, \]
thus $\norm{w_p}$ is uniformly bounded in $p$. 

This shows that there exists $w \in D^{1, 2}(\rn)$ such that, after passing to a subsequence, $w_p \rightharpoonup w$ weakly in $D^{1,2}(\R^N)$ and 
\begin{align}\label{wpconv}
w_p \to w\qquad \text{ in $L^q_{loc}(\rn)$ for any $ q \in [1, 2^*)$.}
\end{align}

Using elliptic regularity we can find $C>0$ such that $\|w_p\|_{L^\infty(\Omega)}\leq C$ uniformly for $p$ close to $1$. Using the barrier constructed in Lemma~\ref{compact:sup:lem}, we can find $R>0$, independent of $p$, such that $\operatorname{supp}(w_p)\subset B_R$ for all $p$ close to $1$. Hence, $\|w_p\|_1
\leq |B_R|^{1-\frac1p}\|w_p\|_p
\leq C$ and therefore $(w_p)$ is uniformly bounded in $X^1$, which implies $w \in X^1$. 
We now observe that, by dominated convergence, 
%see also \cite[p. 18]{ST},  
\[ \int_\Omega |w_p|^p \to \int_\Omega |w|. \]
Thus Fatou's Lemma gives 
\begin{align*} \lim_{p \to 1^+} \mu_p \ge  I_1(w) \ge \mu_1,  \end{align*}
which yields the strong convergence in $D^{1, 2} (\R^N)$. 

To prove that the convergence is actually in $\cC^{1, \alpha}_{loc}$, let $V\subset \rn$ be an open bounded subset.  By \cite[Theorem B.2]{S00}, for any $V' \subset \subset V$ open and for any $q>1$, there exists $C>0$ independent of $p$ such that  
\begin{align}\label{regineq}
    \norm{w_p- w_1}_{W^{2, q}(V')} \le C(\norm{w_p-w_1}_{L^q(V)} + \norm{h_p}_{L^q(V)}),
\end{align}
with
\[ h_p:=Q_\Omega (|w_p|^{p-2}w_p - \sgn(w_1)). \]
By \eqref{wpconv}, it is easy to see that the right hand side in \eqref{regineq} is bounded independently of $p$. Then, $w_p-w_1$ is uniformly bounded in $W^{2, q}(V')$ for any $q >1$. By Sobolev embeddings, this implies that, up to a subsequence, $w_p\to w_1$ in $\cC^{1,\alpha}(V')$. This ends the proof. 
 \end{proof}

\subsection{Asymptotic behavior of nonnegative ground states as 
\texorpdfstring{$p \to 2^-$}{pto2}
}

It is natural to ask what is the asymptotic behavior of solutions as $p\to 2^-$.  This has been already explored for least energy solutions in \cite{CMSS}.  For completeness and for future use, here we collect some of these results.

Let us denote $w_p$ the unique nonnegative minimizer of $I_p$. 
We set 
\[ \alpha_p:= \inf_{\substack{u \in X^p \\ \int_{\R^N} Q_\Omega |u|^p=1}} \int_{\R^N} |\nabla u|^2. \]
It is proved in \cite[Theorem 1.3]{CMSS} that $\alpha_p$ is achieved at some nonnegative function $v_p \in X^p$ for any $p \in (1, 2^*) \setminus \{ 2 \}$, and that $v_p$ is a solution of  
\[ - \Delta v_p = \alpha_p Q_\Omega v_p^{p-1} \quad \text{ in } \R^N. \]
Moreover, a rescaling shows that
\begin{align}\label{rc}
w_p= \alpha_p^{\frac 1{p-2}} v_p.    
\end{align}
By \cite[Theorem 1.2]{CMSS},
\[
\Lambda_1:=\inf \left\{\int_{\R^N} |\nabla v|^2\::\: v \in D^{1,2}(\R^N),\ \int_{\R^N} Q_\Omega |v|^2=1 \right\}>0
\]
and there is an associated eigenfunction $\phi_1$ which is positive in $\rn$.
 
\begin{theorem}[Theorem 3.4 in \cite{CMSS}]\label{q1}
    Let $p \in (1, 2^*)$. Then $v_p \to \phi_1$ in $D^{1, 2}(\R^N)$ as $p\to 2$. Moreover,  
    \[ \lim_{p \to 2} \left( \frac{\Lambda_1}{\alpha_p} \right)^{\frac{1}{2-p}} = e^{-\frac 12 \int_{\R^N} Q_\Omega \phi_1^2 \text{ln}(\phi_1^2)}. \]
\end{theorem}
Using Theorem~\ref{q1} and \eqref{rc} we have that $\supp(w_p)=\supp(v_p)\to \mathbb R^N$. For future use, we state it in the next result.

\begin{theorem}\label{thm:supp}
Let $\Omega\subset \rn$ be an open bounded smooth set and let $w_p$ be the least energy solution of \eqref{sublinear}, then
\begin{align*}
\supp(w_p)\to \R^N\qquad \text{ as }p\to 2^-.
\end{align*}
In particular, for every $R>0$ there is $p_0=p_0(R,\Omega)\in(1,2)$ so that, if $p\in (p_0,2)$, then $B_{R}(0)\subset \supp(w_p)$.
\end{theorem}
\begin{proof}
By Theorem~\ref{q1}, $\alpha_p\to\Lambda_1$ and $v_p\to\phi_1$ in $L^{2^*}(\mathbb R^N)$. Since $-\Delta v_p=\alpha_pQ_\Omega v_p^{p-1}$ in $\R^N$ and $t^{p-1}\leq1+t$ for $t\geq 0$, standard interior regularity estimates and a bootstrap argument on balls show
that $\|v_p\|_{W^{2,q}(B_{R+1})}<C$ for every $R>0$, for some $q>N/2$ and for all $p$ close to 2. Hence, by compact Sobolev embeddings, $v_p\to\phi_1$ uniformly on $\overline{B_R}.$ Since $\phi_1>0$ in $\mathbb R^N$, $m_R:=\min_{\overline{B_R}}\phi_1>0.$ Hence $v_p\geq m_R/2>0$ in $B_R$ for all $p$ sufficiently
close to $2$. Consequently, $B_R\subset\operatorname{supp}(v_p)
=\operatorname{supp}(w_p),$ where the last equality follows from \eqref{rc}.
\end{proof}

\begin{remark}
In fact, using \cite[Theorem 1.3]{CMSS}, one can characterize more precisely the asymptotic behavior of $w_p$ as $p \to 2^-$. In particular,
\begin{enumerate}
    \item if $\Lambda_1=1$, then 
    \[ w_p \to \left( e^{- \frac 12 \int_{\R^N} Q_\Omega \phi_1^2 \ln(\phi_1^2)} \right) \phi_1 \in D^{1,2}(\R^N) \quad \text{ as } p \to 2. \]
\item if $\Lambda_1 >1$, then 
\[ \lim_{p \to 2^-} |w_p|_\infty= 0. \]
\item if $\Lambda_1 < 1$, then 
\[ \lim_{p \to 2^-} w_p(x) = \infty \quad \text{ for a.e. } x \in \R^N. \]
\end{enumerate}
\end{remark}

\subsection{Asymptotic spherical shape of the ground state's support}

\subsubsection{Invariance under polarizations} Let $p\in(1,2)$ and let $w$ be a minimizer of $I=I_p$, where $I_p$ is given by \eqref{Ipdef}.  Let \(H\) be an affine halfspace and let $x^H$ denote the reflection of $x$ with respect to the hyperplane $\partial H$. We define the polarization of $w$ with respect to $H$ by
\[
w^{H}(x):=
\begin{cases}
\max\{w(x),w(x^H)\}, & x\in H,\\[2mm]
\min\{w(x),w(x^H)\}, & x\notin H,
\end{cases}
\]
It is known that
\begin{align}\label{gradeq}
\int_{\mathbb{R}^N} |\nabla w|^2\,dx
=
\int_{\mathbb{R}^N} |\nabla w^H|^2\,dx .   
\end{align}

\begin{lemma}\label{lem:pol}
Let $H$ be a halfspace such that $\Omega\subset H$. Then $w\equiv w^H$ in $\R^N$.    
\end{lemma}
\begin{proof}
Let $H$ be such that $\Omega\subset H$. Let $\omega_1
:=
\{x\in\Omega\::\: w(x)\ge w(x^H)\}$ and $\omega_2
:=
\{x\in\Omega\::\:w(x)<w(x^H)\}.$ Note that $\int_{\Omega}|w^H(x)|^p\,dx
=
\int_{\omega_1}|w(x)|^p\,dx
+
\int_{\omega_2}|w(x^H)|^p\,dx
\geq \int_{\Omega}|w(x)|^p\,dx.$ Then, using that $\int_{\mathbb R^N}|w|^p
=
\int_{\mathbb R^N}|w^H|^p$,
\begin{align*}
\int_{\mathbb R^N}Q_\Omega |w|^p
=2\int_\Omega |w|^p
-
\int_{\mathbb R^N}|w|^p
\leq 
2\int_\Omega |w^H|^p
-
\int_{\mathbb R^N}|w^H|^p=\int_{\mathbb R^N}Q_\Omega |w^H|^p .    
\end{align*}
Together with \eqref{gradeq}, this implies that $I(w)\ge I(w^H).$ The claim follows by the uniqueness of minimizers.
\end{proof}

With Lemma \ref{lem:pol}, the following corollary is a standard consequence.

\begin{corollary}\label{mono:cor}
    Let $R>0$ be such that $\Omega \subset B_R(0)$.  Then $\partial_r w\leq 0$ in $\R^N\backslash B_R(0)$, where $\partial_r$ denotes the derivative with respect to $r=|x|$.
\end{corollary}

\subsubsection{Asymptotic shape}

\begin{proposition}\label{propball}
There is $p_0=p_0(\Omega)\in (1,2)$ so that, for every $p\in (p_0,2)$, there is $r_p>0$ such that 
\begin{align*}
\partial \{w>0\}\subset \{x\in \R^N\::\: r_p\leq |x|\leq r_p+2R\},
\end{align*}
 where $R>0$ is such that $\Omega\subset B_R(0)$ and $r_p\to \infty$ as $p\to 2.$
\end{proposition}
\begin{proof}
Let $R>0$ be such that $\Omega\subset B_R(0)$. By Theorem \ref{thm:supp}, there is $p_0=p_0(R,\Omega)\in(1,2)$ so that $\dist(B_R(0),\partial \{w>0\})>0$ if $p\in (p_0,2)$.  Let $X_0\in \partial \{w>0\}$ be such that 
\begin{align*}
\dist(B_R(0),\partial \{w>0\})=\dist(B_R(0),X_0).    
\end{align*}
 Without loss of generality, we may assume that $X_0=(r,0,\ldots,0)\in \R^N$ with $r=r_p:=\min\limits_{x\in \partial\{w_p>0\}} |x|>R$. By Theorem \ref{thm:supp}, $r_p\to \infty$ as $p\to2^-$.

Next, we construct a family of tangent hyperplanes to the sphere
$\partial B_R(0)$ that we use to reflect the point $X_0$. Let
\[
P(\theta)=R(\cos\theta,\sin\theta,0,\ldots,0)\in\mathbb R^N,\qquad
n(\theta)=(\cos\theta,\sin\theta,0,\ldots,0)\in\mathbb R^N,
\]
where \(n(\theta)\) is the outward unit normal to the sphere
\(\partial B_R(0)\) at \(P(\theta)\). The tangent hyperplane to
\(\partial B_R(0)\) at \(P(\theta)\) is
\[
\Pi(\theta)=\{x\in\mathbb R^N:x\cdot n(\theta)=R\}.
\]

Let \(X(\theta)\) denote the reflection of \(X_0\) across
\(\Pi(\theta)\), that is
\[
X(\theta)
:=
X_0
+
2\bigl(R-X_0\cdot n(\theta)\bigr)n(\theta).
\]
Since \(X_0=(r,0,\ldots,0)\), we have
\[
X_0\cdot n(\theta)=r\cos\theta,
\]
and therefore
\[
X(\theta)
=
(r,0,\ldots,0)
+
2(R-r\cos\theta)
(\cos\theta,\sin\theta,0,\ldots,0).
\]

The squared distance from the origin is
\[
\begin{aligned}
|X(\theta)|^2
&=
|X_0|^2
+
4(R-r\cos\theta)(X_0\cdot n(\theta))
+
4(R-r\cos\theta)^2\\
&=
r^2+4R^2-4R r\cos\theta
\leq r^2+4R r+4R^2=(r+2R)^2
.
\end{aligned}
\]
Using also the definition of $r$, it follows that $r^2
\le
|X(\theta)|^2
\le
(r+2R)^2$ and $r
\le
|X(\theta)|
\le
r+2R.$ Note that $X(\theta_0)=X_0=X(2\pi-\theta_0)$ for $\theta_0=\arccos\!\left(\frac{R}{r}\right)\in\left(0,\frac{\pi}{2}\right).$  Since \(\theta\mapsto X(\theta)\) is continuous and injective on the interval
\([\theta_0,2\pi-\theta_0)\), its image is a closed simple curve, see Figure \ref{fig1}.

Let $H(\theta)$ be a halfspace with $\partial H(\theta) = \Pi(\theta)$ such that $\Omega\subset H(\theta)$ and $X_0\in H(\theta)$.  Then $X(\theta)\in \R^N\backslash H(\theta)$. By Lemma \ref{lem:pol}, $0=w(X_0)\geq w(X(\theta))\geq 0$ for all $\theta\in [\theta_0,2\pi-\theta_0]$. 

Observe that this construction is invariant under rotations that fix $X_0$. We can repeat the argument in every two-dimensional plane containing that axis and the union of the resulting curves is a hypersurface $S$ (see Figure \ref{fig1}) such that $w=0$ on $S$ and $S\subset \{x\in \R^N\::\: r\leq |x|\leq r+2R\}$. By the monotonicity stated in Corollary \ref{mono:cor}, we obtain that $w=0$ in the unbounded component of $\mathbb R^N\backslash S$. Hence, $\partial \{w>0\}$ is between $\partial B_{r}(0)$ and $S\subset B_{r+2R}(0)$. This ends the proof.
\end{proof}

\begin{figure}[h!]
    \centering
    \includegraphics[width=0.25\linewidth]{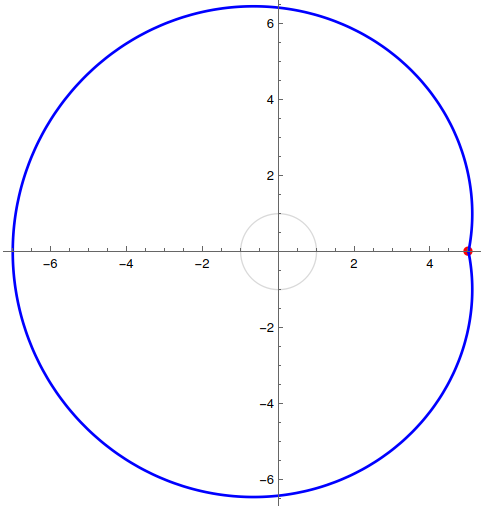}\qquad 
    \includegraphics[width=0.25\linewidth]{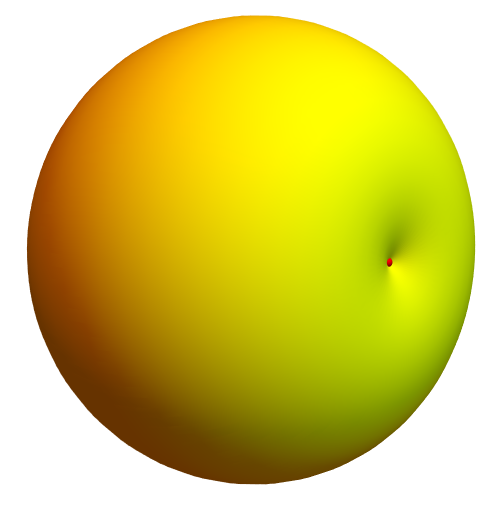}
    \caption{The curve \(\theta\mapsto X(\theta)\) on the left and the hypersurface $S$ on the right. The red dot is the point $X_0$ and the gray circle is $B_R(0)$.}
    \label{fig1}
\end{figure}

\subsection{The structure of the set of nonnegative solutions in the presence of multiple connected components}

In this section, we use some ideas from \cite{BR25} to treat the case when $\Omega$ is not connected. These results characterize the set of nonnegative solutions of \eqref{sublinear} for $p\in [1,2).$

Theorem~\ref{uniquenessOmega} shows that, if $\Omega$ has only one connected component, then \eqref{sublinear} has a unique nonnegative solution. If $\Omega$ has more than one connected component, then either uniqueness or multiplicity can happen, depending on the \emph{distance} between the components.  The following is the analogue of \cite[Theorems 1.1 and 1.2]{BR25} in our setting.

\begin{theorem}\label{thm:solset}
Let $\Omega\subset \rn$ be a bounded smooth open subset with $\ell\in \mathbb N$ connected components denoted by $\omega_i$. Thus,
\begin{align*}
    \Omega = \bigcup_{i=1}^\ell\omega_i.
\end{align*} 
Then \eqref{sublinear} has at least one and at most $2^\ell-1$ nonnegative solutions. Furthermore, $w$ is a nonnegative solution of \eqref{sublinear} if and only if there is a nonempty subset $\mathcal J\subset \{1,\ldots, \ell\}$ such that $w$ is the unique nonnegative solution of 
\begin{align}\label{J}
    -\Delta w = Q_{\Omega_{\mathcal J}} f_p(w),\qquad w \in X^p:=D^{1,2}(\rn)\cap L^{p}(\rn),
\end{align}
with $\Omega_{\mathcal J}:=\bigcup_{i\in \mathcal J}\omega_i,$ satisfying 
\begin{align}\label{Jcond}
 \operatorname{supp}(w)\cap \omega_i\neq\emptyset \quad \text{ if and only if }\quad i\in \mathcal J.
\end{align}
\end{theorem}

\begin{proof}
One nonnegative solution of \eqref{sublinear} is the nonnegative minimizer, see Lemma~\ref{lem:minimizer}.  

If $w$ is a nonnegative solution of \eqref{J} satisfying \eqref{Jcond}, then, since $\operatorname{supp}(w)\cap \omega_i=\emptyset$ for all $i\in \{1,\ldots,\ell\}\backslash  \mathcal J$, $w$ is (trivially) a solution of \eqref{sublinear}. On the other hand, let $w$ be a nonnegative solution of $\eqref{sublinear}$. Then \eqref{Jcond} holds for some $\mathcal J\subset \{1,\ldots, \ell\}$, see also Lemma~\ref{pos:lem}. In particular, $w$ is a nonnegative solution of \eqref{J}.  Since $\operatorname{supp}(w)\cap \omega_i\neq \emptyset$ for all $i\in \mathcal J$, the maximum principle yields that $w>0$ in $\Omega_{\mathcal{J}}$; hence, arguing as in Theorem~\ref{uniquenessOmega}, we have that $w$ is the unique nonnegative solution within this class.  

Finally, the bound $2^\ell-1$ follows easily from the fact that it is the cardinality of the power set of $\{1,\ldots,\ell\}$ without the empty set.
\end{proof}

Let $\operatorname{Isom}(\rn)$ be the group of all isometries of $\rn$ (linear isometries, translations and their compositions).

\begin{theorem}\label{thm:many}
Let $\theta_i\subset \rn$ be a bounded connected smooth open subset for $i=1,\ldots,\ell,$ $\ell\in\mathbb N$, and let $\Omega:=\bigcup_{i=1}^\ell \omega_i$, where $\omega_i:=g_i\theta_i$ for some $g_i\in \operatorname{Isom}(\rn)$. There is $d=d(p,\theta_1,\ldots,\theta_\ell)>0$ so that, if $\dist(\omega_i,\omega_j)>d$ for $i\neq j$, then \eqref{sublinear} has exactly $2^\ell-1$ different nonnegative solutions given by
$$w_{\cJ}:=\sum_{i\in\cJ} w_i,\qquad \emptyset\neq\cJ\subset \{1,\ldots, \ell\},$$
where $w_i$ is the unique nonnegative solution of 
    \begin{align}\label{omegai}
    -\Delta w_i = Q_{\omega_i} f_p(w_i),\qquad w_i \in X^p:=D^{1,2}(\rn)\cap L^{p}(\rn).
\end{align}
The least energy nonnegative solution of \eqref{sublinear} is $\sum_{i=1}^\ell w_i$.
\end{theorem}
\begin{proof}
This is a consequence of Theorem~\ref{thm:solset}. Indeed, for every $i\in \{1,\ldots, \ell\}$, by Theorem~\ref{uniquenessOmega}, there is a unique nonnegative solution $v_i$ of 
    \begin{align}\label{thetai}
    -\Delta v_i = Q_{\theta_i} f_p(v_i),\qquad v_i \in X^p:=D^{1,2}(\rn)\cap L^{p}(\rn).
\end{align}

 Recall that the diameter of a set $A\subset \rn$ is given by 
\begin{align*}
    \operatorname{diam}(A):=\sup\{|x-y|\::\: x,y\in A\}.
\end{align*}
Let $R_i$ be the diameter of $\supp(v_i)$, which is finite by Lemma~\ref{compact:sup:lem}. Clearly, the value of $R_i$ does not change if we consider \eqref{thetai} with $g\theta_i$ instead of $\theta_i$ for any $g\in \operatorname{Isom}(\rn)$. Let $R:=\max_{i} R_i$ and let $d>2R$, then $w_i:=v_i\circ g_i^{-1}$ is a solution of \eqref{J} such that
    \eqref{Jcond} holds for $\mathcal J=\{i\}$ for every $i\in \{1,\ldots, \ell\}$.  Once this is established, then for every nonempty $\mathcal J$ in the power set of $\{1,\ldots,\ell\}$ we have that $\sum_{i\in \mathcal J}w_i$ is the unique nonnegative solution of \eqref{J} satisfying \eqref{Jcond}. As a result, we obtain $2^\ell-1$ solutions.
 \end{proof}

\begin{theorem}\label{thm:one}
Let $\Omega\subset \rn$ be a smooth bounded open subset of $\rn$. There is $p_0=p_0(\Omega)\in(1,2)$ so that, if $p\in (p_0,2)$, then \eqref{sublinear} has a unique nonnegative solution. 
\end{theorem}
\begin{proof}
This is a direct consequence of Theorems~\ref{thm:solset} and~\ref{thm:supp}. 
\end{proof}

\section{Nodal solutions}\label{sec:nod}

In this section we explore the existence and qualitative properties of nodal solutions of \eqref{sublinear}. First we show the existence of a mountain pass solution and then we obtain a least energy nodal solution.  Following \cite{BFMST18}, we show that these two solutions are different in some dumbbell domains.

\subsection{A mountain-pass solution}

For $p\in[1,2)$, let $w$ be the unique nonnegative minimizer of $I_p$, set
\[ \Gamma:=\{ \sigma \in \cC^0([0, 1], X^p):\, \sigma(0)=-w, \, \sigma(1)=w \} \]
and define
\[ c_{mp}:=\inf_{\sigma \in \Gamma} \sup_{t \in [0,1 ]} I_p(\sigma(t)). \]

\begin{lemma}\label{c mp <0}
$c_{mp} <0$. 
\end{lemma}

\begin{proof}
Let $\o_0$ be a connected component of $\o$ and, as usual, let $D^{1,2}_0(\o_0)$ be the closure of $\cC^\infty_c(\o_0)$ in $D^{1,2}(\rn)$. Since $\o_0$ is bounded, $D^{1,2}_0(\o_0)\subset X^p$. The functional $I_{p,0}:D^{1,2}_0(\o_0)\to\r$ is given by 
\[ I_{p,0}(u):=\frac 12 \int_{\Omega_0} |\nabla u|^2 - \frac 1 p \int_{\Omega_0} |u|^p.  \]

\underline{\emph{Case $p>1$.}}
As shown in \cite[Lemma 4.5]{BMPTW}, $I_{p,0}$ has a unique positive minimizer $w_0$ with $I_{p,0}(w_0)<0$ and  there exists a continuous path $\sigma: [0, 1] \to D^{1,2}_0(\Omega_0)$ such that $I_{p,0}(\sigma(t)) <0$ for all $t \in [0, 1]$, $\sigma(0)=-w_0$ and $\sigma(1)=w_0$. 

Let $\gamma: [0, 1] \to X^p$ be the path from $w_0$ to $w$ given by
\[ \gamma(t)=((1-t)w_0^p + tw^p)^{1/p}, \]
and $\eta:[0,1]\to X^p$ be the path from $-w$ to $-w_0$ defined as $\eta(t):=-\gamma(1-t)$. Then, by Lemma~\ref{lem:2},
\[ I_p(\gamma(t)) \le (1-t)I_p(w_0) + tI_p(w) =(1-t)I_{p,0}(w_0) + tI_p(w) <0  \]
and $I_p(\eta(t))=I_p(\gamma(1-t))<0$ for every $t\in[0,1]$. 
Composing the paths $\eta$, $\sigma$ and $\gamma$ we obtain a continuous path connecting $-w$ to $w$ and such that $I_p$ is strictly negative along the path, whence $c_{mp} <0$.

\underline{\emph{Case $p=1$.}}
    We need to check that the arguments in \cite{BMPTW} can be adapted to this case. Let $\varphi_1$ be the positive, $L^2$–normalized eigenfunction associated to $\lambda_1(\Omega_0)$, the first Dirichlet eigenvalue of the Laplace operator in $\Omega_0$, namely
    \[ \int_{\Omega_0}|\nabla \varphi_1|^2 =\lambda_1(\Omega_0) \int_{\Omega_0} |\varphi_1|^2 = \lambda_1(\Omega_0). \]
    Then there exists $\varepsilon >0$ small enough such that 
    \[ I_{1,0}(\varepsilon \varphi_1) = \frac{\varepsilon^2}{2} \int_{\Omega_0} |\nabla \varphi_1|^2 - \varepsilon \int_{\Omega_0} \varphi_1 <0. \]
    We define 
    \[ \zeta(t):=(1-t) \varepsilon \varphi_1 + t w. \]
    Then
    \begin{align*} I_1(\zeta(t)) &\le \frac{1-t}{2} \varepsilon^2\int_{\Omega_0} |\nabla \varphi_1|^2 + \frac t 2 \int_{\R^N} |\nabla w|^2 - (1-t) \varepsilon \int_{\Omega_0} \varphi_1 - t\int_{\R^N} Q_\Omega |w| \\
    &= (1-t) I_{1, 0} (\varepsilon \varphi_1)  + t I_1(w) <0. \end{align*}

    We now recall from \cite{BMPTW} that there exists a continuous path $c(t): [0, 1] \to D^{1,2}_0(\Omega_0)$ connecting $-\varphi_1$ to $\varphi_1$, and such that
    \[ \int_{\Omega_0} |\nabla c(t)|^2 \le \lambda_2(\Omega_0)(1+\delta/2), \quad \int_{\Omega_0} |c(t)|^2=1,  \]
    for some positive $\delta$, where $\lambda_2(\Omega_0)$ is the second Dirichlet eigenvalue of the Laplace operator on $\Omega_0$.
    Observe that, up to choosing a smaller $\varepsilon>0$, one has
    \begin{align*} I_1 (\varepsilon c(t)) = I_{1, 0}(\varepsilon c(t)) &= \frac{\varepsilon^2}{2} \int_{\Omega_0} |\nabla c(t)|^2 - \varepsilon \int_{\Omega_0} |c(t)| \\
    & \le \frac{\varepsilon^2}{2} \int_{\Omega_0} |\nabla c(t)|^2 - \frac{1+\delta}2\lambda_2(\Omega_0) \varepsilon^2   \le -\frac \delta 4 \varepsilon^2 \lambda_2(\Omega_0) <0. \end{align*}

Thus, combining the paths $-\zeta(1-t)$, $c$ and $\zeta$, we get a continuous path connecting $-w$ and $w$ such that $I_1$ is strictly negative along the path, which yields the conclusion.   
\end{proof}

We now want to prove that $c_{mp}$ is a critical value of $I_p$ in a suitable sense. Notice that if $p=1$ the functional $I_p$ is not of class $\cC^1$, thus we use a suitable adaptation of the Mountain Pass Theorem given by \cite[Theorem 3]{Shi}, see also \cite{chang}. 
The Palais-Smale condition for a non-smooth function is stated in \cite[Definition 2]{chang}. It coincides with the usual one \cite{w} if the function is $\cC^1$. For $I_1$ it reads as follows.

\begin{definition}\label{def ps p1}
A sequence $u_n \in X^1$ such that $I_1(u_n) \to c$ and
 \begin{equation}\label{ps cond} \min_{g \in \partial I_1 (u_n)} \norm{g}_{(X^1)^*} \to  0
 \end{equation}
is called a $(PS)_c$-sequence for $I_1$, and $I_1$ is said to satisfy the $(PS)_c$-condition if every $(PS)_c$-sequence for $I_1$ admits a convergent subsequence. 
\end{definition}

\begin{lemma} \label{lem:ps} 
$I_p$ satisfies the $(PS)_c$-condition for every $c<0$.
\end{lemma}
\begin{proof}
\underline{\emph{Case $p>1$.}} 
To show that $I_p$ satisfies the $(PS)_c$-condition for every $c<0$, we need to show that every sequence $(u_k)$ in $X^p$ such that 
$$I_p(u_k) \to c<0\qquad\text{and}\qquad I_p'(u_k)\to 0\text{ \ in \ }(X^p)'$$
contains a convergent subsequence.  Let $(u_k)$ be as above. We may assume that $I_p(u_k)<0$ for all $k$. Then, by Lemma~\ref{bounded Xp}$(ii)$, $(u_k)$ is bounded in $X^p$. Hence, after passing to a subsequence, there exists $u \in D^{1,2}(\R^N) $ such that $u_k \rightharpoonup u$ weakly in $D^{1, 2}(\R^N)$, $u_k \to u$ in $L^p_{loc}(\R^N)$ and a.e. in $\rn$. Notice that $u \in X^p$ since
by Fatou's Lemma 
\[ \int_{\R^N} |u|^p \le \lim_{k \to \infty} \int_{\R^N} |u_k|^p \le \lim_{k \to \infty} \norm{u_k}_{X^p}^p \le C. \]
It follows that, for any $\varphi \in \mathcal{C}_c^\infty(\R^N)$,
\begin{align*} 0 &= \lim_{k\to\infty} I_p'(u_k) \varphi = \lim_{k\to\infty} \left( \int_{\R^N} \nabla u_k\cdot \nabla \varphi - \int_{\R^N} Q_\Omega |u_k|^{p-2} u_k \varphi  \right) \\
& = \int_{\R^N} \nabla u\cdot \nabla \varphi - \int_{\R^N} Q_\Omega |u|^{p-2} u \varphi. 
\end{align*}
This shows that $u$ is a solution to \eqref{sublinear}.
Therefore, using Fatou's lemma,
\begin{align}\label{conv in X}
\nonumber \norm{u}^2&\le \lim_{k\to\infty} \norm{u_k}^2 = \lim_{k\to\infty} \irn Q_\Omega |u_k|^p = \lim_{k\to\infty} \io |u_k|^p-\lim_{k\to\infty} \int_{\rn\smallsetminus\o}|u_k|^p \\
&\le \io |u|^p- \int_{\rn\smallsetminus\o}|u|^p=\int_{\R^N} Q_\Omega |u|^p = \norm{u}^2.
\end{align}
As a consequence, every inequality in \eqref{conv in X} is an equality, which proves that 
$u_n \to u$ in $D^{1, 2}(\R^N)$, and 
\[ \lim_{k \to \infty} \int_{\R^N} |u_k|^p = \int_{\R^N} |u|^p. \]
Since $u_k \to u$ a.e. in $\R^N$, we can now apply Brezis--Lieb Lemma \cite[Lemma 1.32]{w} to conclude that $u_k \to u$ in $L^p(\R^N)$, thus $u_k \to u$ in $X^p$.

\underline{\emph{Case $p=1$.}} To show that $I_1$ satisfies the $(PS)_c$-condition for every $c < 0$ we use Definition~\ref{def ps p1}. If \eqref{ps cond} holds, then there exists $s_n \in S(u_n)$ such that 
\begin{equation}\label{ps cons} \lim_{n \to \infty} \left( \int_{\R^N} \nabla u_n \nabla \varphi - \int_{\R^N} Q_\Omega s_n \varphi \right) =0 \quad \text{for any $\varphi \in X^1$ such that $\norm{\varphi}_{X^1}=1$.} \end{equation}
By Lemma~\ref{bounded Xp}, $u_n$ is bounded in $X^1$, thus up to a subsequence $u_n \rightharpoonup u$ in $D^{1, 2}(\R^N)$ and $u_n \to u$ a.e. in $\rn$.

Since $\|s_n\|_\infty\leq1$, the Banach Alaoglu theorem implies that, up to a subsequence, there is $s\in L^\infty(\mathbb R^N)$ such that $s_n\rightharpoonup^*s$ in $L^\infty(\mathbb R^N)$ and $\|s\|_\infty\leq 1$. Moreover, $s\in S(u)$. Indeed, for almost every $x\in \{u>0\}$ one has that
$u_n(x)>0$ for $n$ sufficiently large, and hence $s_n=1$; similarly, $s_n=-1$ in $\{u<0\}$. By the weak-$*$ convergence of $s_n$ to $s$ we obtain that
\[
s=1\quad\text{a.e. on }\{u>0\},
\qquad
s=-1\quad\text{a.e. on }\{u<0\}.
\]
On $\{u=0\}$, the inequality $|s|\leq1$ follows from
$\|s\|_\infty\leq1$; namely, $s\in S(u)$.  Since $Q_\Omega\varphi\in L^1(\mathbb R^N)$ for every
$\varphi\in X^1$, $\int_{\mathbb R^N}Q_\Omega s_n\varphi
\longrightarrow
\int_{\mathbb R^N}Q_\Omega s\varphi$ as $n\to\infty.$ Passing to the limit in \eqref{ps cons},
\[
\int_{\mathbb R^N}\nabla u\cdot\nabla\varphi
=
\int_{\mathbb R^N}Q_\Omega s\varphi
\qquad\text{for every }\varphi\in X^1.
\]

Moreover, since $u_n$ is bounded in $X^1$, taking $\varphi = u_n / \norm{u_n}_{X^1}$ in \eqref{ps cons}, we get  
\[ \lim_{n \to \infty} \left( \int_{\R^N} |\nabla u_n|^2 - \int_{\R^N} Q_\Omega s_n u_n \right)= 0.\]
Hence, recalling that $s_n \in S(u_n)$ and $s \in S(u)$, and applying Fatou's Lemma, we obtain  
\begin{align*} \norm{u}^2 &\le \lim_{n \to \infty} \norm{u_n}^2 = \lim_{n \to \infty} \int_{\R^N} Q_\Omega s_n u_n = \lim_{n \to \infty} \int_{\R^N} Q_\Omega |u_n| \\
& \le \int_{\R^N} Q_\Omega |u|  = \int_{\R^N} Q_\Omega su= \norm{u}^2. 
\end{align*}
As in the case $p>1$, the Brezis--Lieb Lemma \cite[Lemma 1.32]{w} yields $u_n \to u$ strongly in $X^1$, concluding the proof.
\end{proof}

\begin{theorem}\label{c mp attained}
Assume $\Omega$ is connected. 
Let $p\in [1,2)$, then $\mu_p < c_{mp} <0$ and there exists $u \in X^p$ such that $I_p(u)=c_{mp}$. Moreover, any critical point achieving $c_{mp}$ is a sign-changing solution to \eqref{sublinear}. 
\end{theorem}
\begin{proof}
\underline{\emph{Case $p>1$.}}
    Fix $\varepsilon \in (0, 2 \norm{w}_{X^p})$. We first show that 
    \begin{equation} \label{eq:mp}
        \inf \{ I_p(u): \, \|u-w \|_{X^p} = \varepsilon\} > \mu_p.
    \end{equation}
Arguing by contradiction, assume there exist $u_n \in X^p$ such that 
\[ I_p(u_n) \to \mu_p \quad \text{ and }\quad  \norm{u_n - w}_{X^p}=\varepsilon. \]  
Then, Ekeland's variational principle \cite[Theorem 2.4]{w}, yields $v_n\in X^p$ such that
$$I_p(v_n)\to\mu_p,\qquad I_p'(v_n)\to 0\text{ \ in \ }(X^p)'\qquad\text{and}\qquad\|u_n-v_n\|_{X^p}\to 0.$$
By Lemma~\ref{lem:ps}, after passing to a subsequence, $v_n\to v$ strongly in $X^p$, where $v$ is a minimizer for $I_p$. By Lemma
\ref{lem:sign of minimizer}, and using the fact that $\Omega$ is connected, we conclude that $v \ge 0$ or $v \le 0$. By
Theorem~\ref{uniquenessOmega}, 
either $v=w$ or $v=-w$. If $v_n\to w$, then
$$o(1)=\|u_n-v_n\|_{X^p}+\|v_n-w\|_{X^p}\geq \|u_n-w\|_{X^p}=\eps>0.$$
Hence, $v_n\to -w$. But then
$$o(1)=\|u_n-v_n\|_{X^p}+\|v_n+w\|_{X^p}\geq \|u_n+w\|_{X^p},$$
which implies that
$$o(1)+\eps=\|u_n+w\|_{X^p}+\|u_n-w\|_{X^p} \geq 2\|w\|_{X^p},$$
contradicting our choice of $\eps$. This proves \eqref{eq:mp}.

Since
\[ 0 > c_{mp} \ge \inf \{ I_p(u): \, \norm{u-w}_{X^p} = \varepsilon\} > \mu_p=I_p(\pm w), \]
Lemma~\ref{lem:ps} says that $I_p$ satisfies the $(PS)_{c_{mp}}$-condition, and the Mountain Pass Theorem \cite[Theorem 2.9]{w} yields a critical point $u\in X^p$ of $I_p$ such that $I_p(u)=c_{mp}$. 

Finally, Theorem~\ref{uniquenessOmega} states that $u$ must change sign. 

\underline{\emph{Case $p=1$.}}
Recall the definition of critical point $I_1$ given in Definition~\ref{crit point}.  Fix $\varepsilon \in (0, 2 \norm{w}_{X^1})$. We show that 
    \[ \inf \{ I_1(u): \norm{u-w}_{X^1} = \varepsilon \} > \mu_1. \]
    Assume by contradiction that there exists $u_n \in X^1$ such that 
    \[ I_1(u_n)\to \mu_1 \quad \text{ and } \quad \norm{u_n - w}_{X^1} =\varepsilon. \]
    Then up to a subsequence, $u_n \rightharpoonup u $ in $D^{1, 2}(\R^N)$, strongly in $L^1_{loc}(\R^N)$, and a.e. in $\R^N$. Thus
    \[ \mu_1 = \lim_{n \to \infty} I_1(u_n) \ge I_1(u) \ge \mu_1, \]
    and Brezis--Lieb Lemma \cite[Lemma 1.32]{w}
     implies strong convergence in $X^1$. Moreover, $u$ is a minimizer of $I_1$, and by uniqueness of the minimizers (see Lemma~\ref{lem:sign of minimizer} and Theorem~\ref{uniquenessOmega}), we conclude that $u= \pm w$. However, this contradicts $\norm{u_n - w}_{X^1} = \varepsilon$. 

    Using Lemma~\ref{lem:ps}, we now apply \cite[Theorem 3]{Shi} to infer that there exists a critical point $u \in X^1$ of $I_1$ at level $c_{mp}$, namely $0 \in \partial I_1(u)$ and $I_1(u)=c_{mp}$.  

    The fact that $u$ is a sign-changing solution of \eqref{sublinear} follows by Lemma~\ref{opt sol} and Theorem~\ref{uniquenessOmega}. 
\end{proof}

\subsection{A least-energy nodal solution}

In the previous section we showed that if $\o$ is connected the set of nodal solutions to \eqref{sublinear} is nonempty.  Now we show that there is one nodal solution that has least energy. If $p>1$, the least energy level is given by
\begin{align*}
c_{nod}:=\inf\{I_p(u):u\in X^p, \ I_p'(u)=0, \ u^+\neq0, \ u^-\neq0\}.    
\end{align*}
On the other hand, if $p=1$, the least energy level is stated in terms of the  generalized gradient, namely, 
\begin{align*}
c_{nod}:= \inf \{ I_1(u): \, u \in X^1, \, 0 \in \partial I_1(u),\, u^+ \ne 0, \, u^- \ne 0 \}.    
\end{align*}

In any case, note that $\mu_p\leq c_{nod}\leq c_{mp}<0$ for all $p\in [1,2)$.

\begin{theorem} \label{thm:c_nod}
Assume that $\o$ is connected. Then $\mu_p < c_{nod}$, and $c_{nod}$ is attained.
\end{theorem}

\begin{proof}
\underline{\emph{Case $p>1$.}}
Let $u_k\in X^p$ be such that
\begin{equation}\label{eq:minimizing nodal}
I_p(u_k)\to c_{nod},\quad I_p'(u_k)=0,\quad u_k^+\neq0\text{ \ and \ }u_k^-\neq0.
\end{equation}
By Lemma~\ref{lem:ps}, after passing to a subsequence, $u_k\to u$ strongly in $X^p$. Therefore, $I_p(u)=c_{nod}$ and $I_p'(u)=0$. Thus, if $u^+\neq0$ and $u^-\neq0$, then $c_{nod}$ is attained and, by Lemma~\ref{lem:sign of minimizer}, $c_{nod}>\mu_p$.

To show that $u^+\neq0$ and $u^-\neq0$ we argue by contradiction. Assume that $u^-=0$.  By the regularity bound in Lemma~\ref{Linfty:lem}, we know that 
\begin{align*}
    u_k\to u\quad \text{in \ }L^\infty(\Omega).
\end{align*}
By Lemma~\ref{pos:lem}, 
\begin{align}\label{pos:rmk}
\text{there exists $\vr>0$ such that $u(x)\geq 2\vr$ for all $x\in\bar\o$. }    
\end{align}
Let $k_0\in\n$ be such that $|u_k(x)-u(x)|<\vr$  for all $x\in\bar\o$ and all $k\geq k_0$. Then $u_k(x)\geq\vr>0$ for all $x\in\bar\o$ and all $k\geq k_0$. Since $0=I_p'(u_k)u_k^-=I_p'(u_k^-)u_k^-$, it follows that
$$\|u_k^-\|^2=\irn Q_\o |u_k^-|^p=-\irn|u_k^-|^p\leq 0\qquad\text{for all \ }k\geq k_0.$$
Therefore, $u_k^-=0$ for all $k\geq k_0$, contradicting \eqref{eq:minimizing nodal}. The other case is similar.

\underline{\emph{Case $p=1$.}} Let $u_k\in X^1$ be such that
\begin{equation*}
I_1(u_k)\to c_{nod},\quad 0\in\partial I_1(u_k),\quad u_k^+\neq0\text{ \ and \ }u_k^-\neq0.
\end{equation*}
Then $(u_k)$ is a $(PS)_{c_{nod}}$-sequence for $I_1$ in the sense of Definition~\ref{def ps p1}, and, by Lemma~\ref{lem:ps}, it admits a subsequence that converges in $X^1$. Moreover, since $0 \in \partial I_1(u_k)$, if $u_k(x) \ge \vr >0$ for all $x \in \overline \Omega$ and all $k \ge k_0$, then 
we have  
\[ \|u_k^-\|^2 = \int_{\R^N} Q_\Omega |u_k^-| = - \int_{\R^N} |u_k^-| \le 0.  \]
The proof is completed as in the case $p>1$. 
\end{proof}

\begin{remark}
If $\Omega$ is not connected, then it may happen that $c_{nod}=\mu_p$ as a consequence of Lemma~\ref{compact:sup:lem}; for instance, if $\Omega$ is the disjoint union of two balls far apart (so that the supports of the solutions at each ball does not intersect), then $\mu_p$ is achieved at a function which is positive on one ball and negative on the other.  

On the other hand, if $\Omega$ is not connected but it has a unique nonnegative solution (see, for instance, Theorems~\ref{thm:solset} and ~\ref{thm:one}), then one can still guarantee \eqref{pos:rmk} and show the existence of a least energy nodal solution in this case too (which could be a minimizer, as noted in Remark~\ref{rem:sign-changing}).
\end{remark}

\subsection{Mountain pass vs. least energy solutions}\label{sec:dumbbell}

Following \cite{BFMST18}, we now show that the mountain pass solution is different from the least energy nodal solution in some domains.

For fixed $p\in[1,2),$ we now write
\begin{align*}
I_\Omega(u):= \frac 12 \norm{u}^2 - \frac 1p \int_{\R^N} Q_\Omega |u|^p    
\end{align*}
in order to emphasize the dependence on $\Omega$. We use $\mu(\Omega)$ to denote  the least-energy level of $I_\Omega$ in $X^p$. Notice that, if $\Omega_1 \subseteq \Omega_2$, then, for any $u \in X^p$,
\begin{equation}\label{monot mu} 
I_{\Omega_1}(u) \ge I_{\Omega_2}(u) \quad \text{ and } \quad \mu(\Omega_1) \ge \mu(\Omega_2). 
\end{equation}

Let 
\begin{align*}
\Omega_0:=B_1(x_1) \cup B_1(x_2),    
\end{align*}
 where $x_1$ and $x_2$ are such that $B_1(x_1)$ and $B_1(x_2)$ have distance larger than the value $d(p, B_1(0),B_1(0))$ given by Theorem~\ref{thm:many}. We denote by $w_1$ the nonnegative minimizer of $I_{B_1(x_1)}$ and by $w_2$ the nonnegative minimizer of $I_{B_1(x_2)}$ (notice that one can be obtained  from the other one by translation). 

We consider the dumbbell domain $\Omega_\delta$ obtained by connecting in a smooth way $B_1(x_1)$ and $B_1(x_2)$ with a thin tube of width $\delta$ (to be chosen later). 
\begin{lemma}\label{lem:limit delta}
  One has $\mu(\Omega_\delta) \to \mu(\Omega_0)$ as $\delta \to 0$. Moreover, if $w_\delta$ is the positive optimizer of $\mu(\Omega_\delta)$, then $w_\delta \to w_1+w_2$ in $X^p$ as $\delta \to 0$.   
\end{lemma}
\begin{proof}
    We notice that, for any $u \in X^p$ and any $\delta >0$, since $\Omega_0 \subset \Omega_\delta$, 
\[ \mu(\Omega_\delta) \le \mu(\Omega_0). \]
Moreover, by H\"older and Sobolev inequalities, there exist $C, \tilde C>0$ independent of $\delta$ such that 
\begin{equation}\label{dis mu} \mu(\Omega_0) \ge \mu(\Omega_\delta) = I_{\Omega_\delta}(w_\delta) \ge \frac 12 \norm{w_\delta}^2 - \frac C p |\Omega_\delta|^{(2^*-p)/2^*} \norm{w_\delta}^p \ge \frac 12 \norm{w_\delta}^2 - \tilde C \norm{w_\delta}^p,  \end{equation}
whence $w_\delta$ is uniformly bounded in $D^{1, 2}(\R^N)$. Notice that an inspection of Lemma~\ref{bounded Xp} shows that $w_\delta$ is also uniformly bounded in the $L^p(\R^N)$ norm.
Thus, there exists $\hat w \in X^p$ such that $w_\delta \rightharpoonup \hat w$ in $D^{1, 2}(\R^N)$ and $w_\delta \to \hat w$ in $L^p_{loc}(\R^N)$.

On the other hand, by the H\"older and the Sobolev  inequalities, there exist $C, \tilde C>0$ independent of $\delta$ such that 
\[ \int_{\Omega_\delta \setminus \Omega_0} |w_\delta|^p \le C\, |\Omega_\delta \setminus \Omega_0|^{(2^*-p)/2^*} \norm{w_\delta}^p \le \tilde C \, |\Omega_\delta \setminus \Omega_0|^{(2^*-p)/2^*} \to 0, \quad \text{ as } \delta \to 0;\]
hence, 
\[ \int_{\Omega_\delta} |w_\delta|^p = \int_{\Omega_0} |w_\delta|^p + \int_{\Omega_\delta \setminus \Omega_0} |w_\delta|^p \to \int_{\Omega_0}|\hat w|^p. \]
Therefore, by \eqref{dis mu} and Fatou's Lemma, 
\[ \mu(\Omega_0) \ge \liminf_{\delta \to 0} I_{\Omega_\delta}(w_\delta) \ge I_{\Omega_0} (\hat w) \ge \mu(\Omega_0). \]
This implies $\mu(\Omega_\delta) \to \mu(\Omega_0)$. Moreover, $w_\delta \to \hat w$ in $D^{1, 2}(\R^N)$ and 
\[ \int_{\R^N} |w_\delta|^p = \int_{\R^N \setminus \Omega_\delta} |w_\delta|^p + \int_{\Omega_\delta} |w_\delta|^p  \to \int_{\R^N \setminus \Omega_0} | \hat w |^p + \int_{\Omega_0} |\hat w |^p = \int_{\R^N} |\hat w|^p, \]
so that the Brezis-Lieb Lemma \cite[Lemma 1.32]{w} implies 
$w_\delta \to \hat w$ in $X^p$. By Theorem~\ref{thm:many}, $\hat w = w_1+w_2$.
\end{proof}

We fix 
\begin{align}\label{epsstar}
    \varepsilon_*:=\norm{w_1}=\norm{w_2}.
\end{align}
Since $w_1, w_2$ have disjoint supports, $\norm{w_1\pm w_2}=2 \varepsilon_*$. 
    Then, there exists $\delta_0>0$ such that 
    \begin{equation}\label{estimate norm w delta} \norm{w_\delta} > \varepsilon_*
    \quad \text{ and }\quad \norm{w_\delta - (w_1-w_2)} > \varepsilon_*
    \qquad \text{for any $\delta \in (0, \delta_0)$}.
    \end{equation}
\begin{lemma}\label{lem:bound cmp}
Let $\delta_0$ be as in \eqref{estimate norm w delta} and let $\eps^*$ be as in \eqref{epsstar}. There exists $\delta_1 < \delta_0$ and $c > \mu(\Omega_0)$ such that, for any $\delta \in (0, \delta_1)$, 
    \[ I_{\Omega_\delta}(v) \ge c \quad \text{ for any } v \in X^p, \quad \norm{v - w_\delta}=\varepsilon_* \]
    and
     \[ I_{\Omega_\delta}(v) \ge c \quad \text{ for any } v \in X^p, \quad \norm{v + w_\delta}=\varepsilon_*. \]
\end{lemma}
\begin{proof}
    We only prove the first statement, as the second one is similar. By contradiction, let us assume that there exists $\delta_k \to 0$ and $v_k:=v_{\delta_k} \in X^p$ such that $\norm{v_k - w_{\delta_k}}=\varepsilon_*$ and $I_{\Omega_{\delta_k}}(v_k) \to \mu(\Omega_0)$ as $k \to \infty$. 
    Then, as in Lemma~\ref{lem:limit delta}, by Lemma~\ref{bounded Xp}, $v_k$ is bounded in $X^p$, and there exists $v \in X^p$ such that $v_k \rightharpoonup v$ in $D^{1, 2}(\R^N)$ and $v_k \to v$ in $L^p_{loc}(\R^N)$. 
    Since
    \[ \mu(\Omega_0) \le I_{\Omega_0} (v)\le  \lim_k I_{\Omega_{\delta_k}} (v_k) = \mu(\Omega_0), \]
    then $v_k \to v$ in $D^{1, 2}(\R^N)$, which implies by Lemma~\ref{lem:limit delta} that $\norm{v- (w_1+w_2)}=\varepsilon_*$. Moreover, $v$ is an optimizer of $\mu(\Omega_0)$, hence $|v| = w_1+w_2$ by Theorem~\ref{thm:many}. As a consequence,
    \[ \norm{v-(w_1+w_2)} = \begin{cases}
        0, & \text{ if } v=w_1+w_2,\\
        4 \varepsilon_*, & \text{ if } v=-w_1-w_2, \\
        2 \varepsilon_*, & \text{ if } v=w_1-w_2 \text{ or } v = w_2 - w_1. 
    \end{cases}\]
    This yields a contradiction, and concludes the proof.
\end{proof}

\begin{theorem}\label{dumbbell}
Let $\delta_1$ be as in Lemma~\ref{lem:bound cmp}.     If $\delta < \delta_1$, there exists a nodal solution $v$ of 
    \begin{equation}\label{eq:omega delta} - \Delta v= Q_{\Omega_\delta} |v|^{p-2}v \quad \text{ in } \R^N \end{equation}
    such that
    \[ c_{mp}(\Omega_\delta) > I_{\Omega_\delta}(v) \ge c_{nod}(\Omega_\delta). \]
\end{theorem}
\begin{proof}
\underline{ \emph{Case $p>1:$}}     Recall that by \eqref{estimate norm w delta}, $\norm{w_\delta}> \varepsilon_*$ if $\delta< \delta_0$. Therefore, any continuous path connecting $-w_\delta$ and $w_\delta$ intersects $\partial \mathcal{A}_\delta$, where 
    \[ \mathcal{A}_\delta:=\{ v \in X^p:\, \norm{v \pm w_\delta} > \varepsilon_*\}. \]
    Thus, applying Lemma~\ref{lem:bound cmp}, if $\delta < \delta_1$, 
    \[ c_{mp}(\Omega_\delta) \ge \inf_{\partial \mathcal A_\delta} I(v) \ge c > \mu(\Omega_0).  \]
    On the other hand, since $\norm{w_\delta -(w_1-w_2)} > \varepsilon_*$ by \eqref{estimate norm w delta}, and recalling \eqref{monot mu}, 
    \[ \mu(\Omega_0)=I_{\Omega_0}(w_1-w_2) \ge \inf_{\mathcal A_\delta} I_{\Omega_0}(v) \ge \inf_{\mathcal A_\delta} I_{\Omega_\delta}(v)=: \hat c_\delta \ge \mu(\Omega_\delta). \]
    We now show that the infimum $\hat c_\delta $ is attained by a nodal function, whence the conclusion follows. 
Fix $\delta$ and apply Ekeland's variational principle \cite[Theorem 2.4]{w} to get a sequence $v_k \in X^p$ such that $\norm{v_k \pm w_\delta} > \varepsilon_*$,  $I_{\Omega_\delta}(v_k) \to \hat c_\delta$ and $I'_{\Omega_\delta}(v_k) \to 0$. By Lemma~\ref{lem:ps}, there exists $v \in X^p$ such that $v_k \to v$ in $X^p$. Moreover, $I'_{\Omega_\delta}(v)\varphi=0$ for any $\varphi \in \mathcal{C}_c^\infty(\R^N)$, which means that $v$ is a solution to \eqref{eq:omega delta}.
If $v$ does not change sign, since $\Omega_\delta$ is connected, then an application of Theorem~\ref{uniquenessOmega} yields that $|v|$ is the nonnegative optimizer of $\mu(\Omega_\delta)$, and it coincides with $w_\delta$. However, $\norm{v\pm w_\delta} \ge \varepsilon_*$, a contradiction.

\medskip

\underline{ \emph{Case $p=1:$}} The case $p=1$ can be treated similarly. The only delicate part is the application of Ekeland's variational principle.  We want to prove that 
\[ \hat c_\delta:= \inf_{v \in \mathcal{A}_\delta} I_{\Omega_\delta}(v) \]
is attained, where 
\[ \mathcal{A}_\delta:=\{ v \in X^1:\, \norm{v \pm w_\delta} > \varepsilon_*\}. \]
Let us define
\[ F(v):=\begin{cases}
    I_{\Omega_\delta}(v), & \text{ if } v \in \mathcal{A}_\delta, \\
    +\infty, &\text{ otherwise. }
\end{cases}\]
Then,
\[ \hat c_\delta= \inf_{v \in X^1} F(v), \]
and, by Ekeland's variational principle \cite{ek}, there exists a sequence $v_n \in X^1$ and $\varepsilon_n \to 0$ such that 
\[+ \infty > \inf_{v \in X^1} F(v) + \varepsilon_n^2 \ge F(v_n)  \]
and
\[ F(v) > F(v_n) - \varepsilon_n \norm{v_n - v}_{X^1} \quad \text{ for any } v \ne v_n. \]
In particular, $v_n \in \mathcal{A}_\delta$, and
\[ I_{\Omega_\delta}(v_n) \to \hat c_\delta. \]
Moreover, for any $n \in \mathbb{N}$, there exists $\lambda$ small enough such that 
$\norm{v_n +\lambda v \pm w_{\delta}} > \varepsilon_*$ for any $v \in X^1$ such that $ \norm{v}_{X^1}=1$, and 
\[ I_{\Omega_\delta}(v_n + \lambda v) > I_{\Omega_\delta}(v_n) - \varepsilon_n \lambda \quad \text{ for any } v \in X^1 \text{ such that } \norm{v}_{X^1}=1.  \]
Thus
\[ 0 \in \partial I_{\Omega_\delta}(v_n) + \varepsilon_n \overline {B^*}, \]
where $B^*$ is the unit ball in $(X^1)^*$. As a consequence, $v_n $ is a $(PS)_{\hat c}$-sequence in the sense of Definition~\ref{def ps p1}, and it admits a converging subsequence. The rest of the proof follows as in the case $p >1$.
\end{proof}

\subsection{Nodal symmetric solutions}

Let $G$ be a closed subgroup of the group $O(N)$ of linear isometries of $\rn$ and let $\phi:G\to\z/2:=\{1,-1\}$ be a continuous homomorphism of groups. A function $u:\rn\to\r$ is called \emph{$\phi$-equivariant} if it satisfies
$$u(gx)=\phi(g)u(x)\qquad\text{for every \ }g\in G, \ x\in\rn.$$
We write $Gx:=\{gx:g\in G\}$ for the $G$-orbit of the point $x$, and assume from now on that $\o$ is $G$-invariant, that is, that $Gx\subset\o$ for every $x\in\o$. Then, the functional $I_p$ satisfies $I_p(u(gx))=I_p(\phi(g)u(x))$ for every $g\in G$ and $x\in\rn$. By the principle of symmetric criticality for Banach spaces, see \cite[Theorem 2.2]{ko}  and \cite{km} for the case $p=1$, the $\phi$-equivariant solutions to the problem \eqref{sublinear} are the critical points of the restriction of $I_p$ to the space
$$(X^p)^\phi:=\{u\in X^p:u \text{ is }\phi\text{-equivariant}\}.$$
This space might be trivial; for instance, if $G=O(N)$ and $\phi=\det:O(N)\to\z/2$ is the determinant, then the only $\phi$-equivariant function is the trivial one. To avoid this behavior we ask that $\phi$ has the following property:
\begin{itemize}
\item[$(*)$] There exists $x_0\in\o$ such that $Gx_0\neq Kx_0$, where $K:=\ker\phi$.
\end{itemize}
Then $D^{1,2}_0(\o)^\phi:=\{u\in D^{1,2}_0(\o):u \text{ is }\phi\text{-equivariant}\}$ has infinite dimension. Note that $(*)$ implies that $\phi$ is surjective. Thus, every nontrivial $\phi$-equivariant function changes sign. Let
$$\mu_p^\phi:=\inf_{u\in(X^p)^\phi}I_p(u).$$

\begin{lemma} \label{lem:G-minimizer}
\begin{itemize}
\item[$(i)$]$-\infty<\mu_p\leq\mu_p^\phi<0$.
\item[$(ii)$] $\mu_p^\phi$ is attained at a $\phi$-equivariant solution to the problem \eqref{sublinear}.
\end{itemize}
\end{lemma}

\begin{proof}
$(i):$ The inequality $\mu_p\leq\mu_p^\phi$ is obvious.  To prove that $\mu_p^\phi<0$ consider the $\phi$-equivariant Dirichlet eigenvalue problem in $\o$,
$$-\Delta \psi = \lambda \psi,\qquad \psi\in D^{1,2}_0(\o)^\phi.$$
Let 
$$\lambda_1^\phi(\o):=\inf_{u\in D^{1,2}_0(\o)^\phi\smallsetminus\{0\}}\frac{\|u\|^2}{|u|_2^2}$$ 
be the first $\phi$-equivariant eigenvalue and $\psi_1^\phi\in D^{1,2}_0(\o)^\phi$ be the first $\phi$-equivariant eigenfunction that satisfies $|\psi_1^\phi|_2=1$. Fix $\alpha>\frac{p}{2}\lambda_1^\phi(\o)$ and take $\eps>0$ such that $|\eps\psi_1^\phi(x)|^p\geq \alpha|\eps\psi_1^\phi(x)|^2$ for every $x\in\o$. Then,
$$\mu_p^\phi\leq I(\eps\psi_1^\phi)=\frac{1}{2}\|\eps\psi_1^\phi\|^2-\frac{1}{p}\io|\eps\psi_1^\phi|^p\leq \frac{\lambda_1^\phi(\o)}{2}\io|\eps\psi_1^\phi|^2-\frac{\alpha}{p}\io|\eps\psi_1^\phi|^2<0,$$
as claimed.

$(ii):$ The proof of this statement is completely analogous to that of Lemma~\ref{lem:minimizer}$(ii)$.
\end{proof}

\begin{example}\label{ex:groups}Write $\rn=\cc\times\r^{N-2}$ and a point in $\rn$ as $x=(z,y)$ with $z\in\cc$ and $y\in\r^{N-2}$. For every even number $k\in\n$ let
$$\vr_k(x):=(\e^{\frac{2\pi\mathrm{i}}{k}}z,y),\qquad x=(z,y)\in\cc\times\r^{N-2},$$
be the rotation of angle $\frac{2\pi}{k}$ around $\{0\}\times\r^{N-2}$ and let $G_k$ be the group generated by $\vr_k$. Since $k$ is even, the homomorphism $\phi_k:G_k\to\z/2$ such that $\phi_k(\vr_k)=-1$ is well defined. Note that it satisfies assumption $(*)$. A function $u:\rn\to\r$ is $\phi_k$-equivariant if
$$u(\vr_k^n x)=(-1)^n u(x)\qquad\text{for all \ }x\in\rn, \ n=0,\ldots,k-1.$$
\end{example}

\begin{theorem}\label{thm:symm}
If $\o$ is invariant under rotations around $\{0\}\times\r^{N-2}$, that is,
\begin{equation*}
(\e^{\mathrm{i}\theta}z,y)\in\o\qquad\text{for every \ }\theta\in [0,2\pi], \ (z,y)\in\o, \ \text{ \ where \ }z\in\cc, \ y\in\r^{N-2},
\end{equation*}
then the problem \eqref{sublinear} has a sequence $(v_k)$ of sign-changing solutions such that
\begin{itemize}
\item[$(i)$] $v_k$ is $\phi_{2^k}$-equivariant and has least energy among all $\phi_{2^k}$-equivariant solutions to \eqref{sublinear}, where $\phi_{2^k}:G_{2^k}\to\z/2$ is the homomorphism defined in Example~\ref{ex:groups}.
\item[$(ii)$] $v_k\neq v_\ell$ if $k\neq\ell$,
\item[$(iii)$] $v_k\to 0$ strongly in $X^p$.
\end{itemize}
\end{theorem}

\begin{proof}
Let $v_k$ be the $\phi_{2^k}$-equivariant solution given by Lemma~\ref{lem:G-minimizer}$(ii)$. Then $v_k$ satisfies $(i)$.

Let $k<\ell$ and assume that $v_k=v_\ell$. Let $x_0\in\rn$ be such that $v_\ell(x_0)\neq 0$. Since $\vr_{2^\ell}^n=\vr_{2^k}$ with $n:=2^{\ell-k}$ an even number, we have that
$$v_\ell(x_0)=v_\ell(\vr_{2^\ell}^nx_0)=v_k(\vr_{2^k}x_0)=-v_k(x_0)=-v_\ell(x_0),$$
which is a contradiction. Therefore, $(ii)$ holds true.

Next, we prove $(iii)$. Given $0\leq r_1<r_2<\infty$, \  $0\leq\theta_1<\theta_2\leq 2\pi$ and real numbers $a_i<b_i$, we set
$$\Sigma:=\{(r\e^{\mathrm{i}\theta},y)\in\cc\times\r^{N-2}:r\in[r_1,r_2], \ \theta\in[\theta_1,\theta_2], \ y_i\in[a_i,b_i]\},$$
and call it a $(\theta_2-\theta_1)$-sector. Note that, if $\theta_2-\theta_1=\frac{2\pi}{2^m}$ for some $m\in\n$, then, for any $k>m$, $\Sigma$ is the union of an even number of $\frac{2\pi}{2^k}$-sectors  whose interiors are pairwise disjoint. As a consequence, every $\phi_{2^k}$-equivariant function $u$ satisfies
\begin{equation}\label{eq:supports}
|\supp(u^+)\cap\Sigma|=|\supp(u^-)\cap\Sigma|.
\end{equation}
Now, as $I_p(v_k)\in[c_{nod},0)$ and $v_k$ is a solution of \eqref{sublinear}, after passing to a subsequence, we have that $I_p(v_k)\to c\in[c_{nod},0]$. If $c=0$, then, since $\|v_k\|^2=\int_{\rn}Q_\Omega |v_k|^p$,
$$o(1)=I_p(v_k)=\frac{p-2}{2p}\|v_k\|^2.$$
Hence, $v_k\to 0$ strongly in $X^p$, as claimed. If, on the other hand, $c<0$, Lemma~\ref{lem:ps} states that, after passing to a subsequence, $v_k\to v$ strongly in $X^p$ and a.e. in $\rn$. Hence, $v$ is a solution to problem \eqref{sublinear}. Arguing by contradiction, assume that $v\neq 0$. Then, there exists a $\frac{2\pi}{2^m}$-sector $\Sigma$ such that $v(x)\geq\eps>0$ for every $x\in\Sigma$. Fix $\delta\in(0,\frac{1}{2}|\Sigma|)$. By Egorov's theorem \cite[Theorem 7.12]{Bartle}, there is a set $U_\delta$ such that $|U_\delta|<\delta$ and $v_k\to v$ uniformly in $\Sigma\smallsetminus U_\delta$. Thus, there exists $k_0>m$ such that
$$|v_k(x)-v(x)|<\frac{\epsilon}{2}\qquad\text{for all \ }x\in \Sigma\smallsetminus U_\delta\text{ \ and \ }k\geq k_0,$$
and, as a consequence, $v_k(x)\geq\frac{\epsilon}{2}$ for all $x\in \Sigma\smallsetminus U_\delta$ and $k\geq k_0$. From \eqref{eq:supports} we get
$$|\Sigma|\geq |\supp(v_k^+)\cap\Sigma|+|\supp(v_k^-)\cap\Sigma|=2|\supp(v_k^+)\cap\Sigma|\geq 2|\Sigma|-2\delta>|\Sigma|,$$
a contradiction. This proves that $v=0$, as claimed.
\end{proof}

\section{Proof of main results}\label{sec:proofs}
We now collect the results of the previous sections, to give the proof of the main results stated in the Introduction. 

\begin{proof}[Proof of Theorem~\ref{thm:main}]
\emph{Statement 1.} This follows from Theorem~\ref{uniquenessOmega} and Lemma~\ref{pos:ground state}. 

\emph{Statement 2.} The fact that any solution to \eqref{sublinear} has compact support is proved in Lemma~\ref{compact:sup:lem}, whereas the asymptotically spherical shape is a consequence of Proposition~\ref{propball}. 

\emph{Statement 3.} This is a consequence of Theorem~\ref{thm:c_nod} and Theorem~\ref{c mp attained}. An example of a set for which the least energy nodal level and the mountain pass level are different is given in Section~\ref{sec:dumbbell}. 
\end{proof}

\begin{proof}[Proof of Theorem~\ref{thm:starshaped:intro}]
    The claim follows from Theorem~\ref{thm:starshaped} and Corollary~\ref{Cor:starshaped}.
\end{proof}

\begin{proof}[Proof of Theorem~\ref{thm:main:2}]
    \emph{Statement 1.} is a consequence of Theorem~\ref{uniquenessOmega}. \emph{Statement 2.} follows from Theorem~\ref{thm:one}. \emph{Statement 3.} is contained in Theorem~\ref{thm:many}.
\end{proof}

\begin{proof}[Proof of Theorem~\ref{thm:main:3}]
This result is contained in Theorem~\ref{thm:symm}.
\end{proof}

\section{Overdetermined problems}\label{sec:over}

In this section we offer a different perspective on problem \eqref{sublinear} for $p=1$ from the point of view of overdetermined problems.  Let $U\subset \rn$ be an open bounded subset of $\rn$, $N\geq 3,$ and let $D\subset U$. Consider the following overdetermined problem with sign-changing source
\begin{align}\label{odp}
 -\Delta \tau =
\begin{cases}
1 & \text{in \ }D, \\
-1& \text{in \ }U\smallsetminus D,
\end{cases} 
\qquad \tau = \partial_\nu \tau = 0\quad \text{ on }\partial U,\qquad  \tau>0\quad \text{ in }U. 
\end{align}

This equation has several interesting interpretations.  
\begin{enumerate}
    \item $\tau$ can describe a heat distribution in equilibrium in a thermally isolated body $U$ (due to the Neumann boundary conditions), where $D$ is being uniformly heated and $U\backslash D$ is being uniformly cooled. Solving \eqref{odp} amounts to asking if there are shapes $(D,U)$ so that, at thermal equilibrium, the temperature at the boundary $\partial U$ is exactly zero (Dirichlet boundary conditions). 
    \item $\tau$ can describe an electric potential (in electrostatic equilibrium) in an isolated body $U$ (due to the Neumann boundary conditions), where $D$ has a positive charge density of $+1$ and $U\backslash D$ has a negative charge density of $-1$. Solving \eqref{odp} means that  there exist configurations $(D,U)$ so that the electric field at the boundary $\partial U$ is exactly zero (Dirichlet boundary conditions). Indeed, in this case the electric field $E=\nabla \tau$ has zero perpendicular and tangential components at the boundary $\partial U$. 
    \item $\tau$ can describe the vertical deflection of an elastic membrane fixed at the boundary $\partial U$. This membrane is being pulled uniformly upwards in $D$ and downwards in $U\backslash D$. A solution $(D,U)$ of \eqref{odp} means that there is a configuration under which the membrane is flat at $\partial U$ (its gradient is zero at $\partial U$).
\end{enumerate}
Similar interpretations can be done in other settings such as fluid dynamics, elastic materials, and potential theory. 

In this setting, we can distinguish two types of problems.

\medskip

\noindent\textbf{The inner-set problem.}
For a given open bounded set $U\subset \rn$, is there an open bounded set $D \subset \subset U$ so that \eqref{odp} has a solution?

\medskip

\noindent\textbf{The outer-set problem.}
For a given open bounded set $D\subset \rn$, is there an open bounded set $U$ such that $D \subset \subset U$ and \eqref{odp} has a solution?

\medskip

As a corollary of our results, we can give a satisfactory answer to the outer-set problem. 

\begin{theorem}\label{odthm}
For any open bounded smooth set $D\subset \rn$, there is a unique open bounded set $U\subset \rn$ such that $D \subset \subset U$ and \eqref{odp} has a solution.
\end{theorem}
\begin{proof}
Let $D$ be an open bounded set and let $w$ be the unique nonnegative minimizer of \eqref{sublinear} for $p=1$ and $\Omega=D$ (see Lemma~\ref{lem:minimizer} and Theorem~\ref{uniquenessOmega}).  By Lemma~\ref{opt sol} we know that $w$ is a solution of \eqref{sublinear} with $p=1$. Note that $w$ has compact support, by Lemma~\ref{compact:sup:lem}. Let $\tau:=w$ and let $U$ be the set $\{w>0\}$. Then $\tau$ and $U$ yield a solution for \eqref{odp}, because $\sgn(w)=1$ in $U$. By Theorem \ref{thm:solset} there is only one nonnegative solution that is positive in $\Omega$. 
\end{proof}

Our methods do not apply to the inner-set problem, which we leave as an interesting open question.  See also \cite{c1} and the references therein, where the analogue of the inner/outer-set problem has been considered for other two-phase overdetermined problems.

\begin{remark}[Some explicit solutions]\label {sec:exp}
    One of the advantages of the case $p=1$ is that it is easier to do some explicit computations in this setting.  As an example, let $B_1:=\{x\in \rn\::\: |x|<1\}$, and define $w$ as
\begin{align*}
w(r):=
\begin{cases}
\dfrac{1-2^{\frac{2-N}{N}}}{N-2} - \dfrac{1}{2N}r^{2}, & r < 1, \\[6pt]
-\dfrac{2^{\frac{2-N}{N}}}{N-2} 
+ \dfrac{2}{N(N-2)}r^{2-N}
+ \dfrac{1}{2N}r^{2}
, & 1 \le r < 2^{1/N}, \\[10pt]
0, & r \ge 2^{1/N}.
\end{cases}
\end{align*}
% If $N=2,$ let 
% \begin{align*}
%  w(r):=
% \begin{cases}
% \dfrac{\ln 2}{2} - \dfrac{1}{4}r^{2}, & r < 1, \\[6pt]
% \dfrac{\ln 2-1}{2}
% -\ln r
% + \dfrac{1}{4}r^{2}
% , & 1 \le r < \sqrt{2}, \\[10pt]
% 0, & r \ge \sqrt{2}.
% \end{cases}
% \end{align*}
Then a nonnegative solution of the problem 
\begin{align*}
    -\Delta w = Q_{B_1}\sgn(w)\quad \text{ in }\rn,
\end{align*}
is given by $w(x)=w(|x|)$ (this is actually the only nonnegative solution, see Theorem~\ref{uniquenessOmega}).

Notice that by Remark~\ref{scaling} this also gives an explicit solution for the problem
\[ -\Delta v_R = Q_{B_R}\sgn(v_R)\quad \text{ in }\rn \]
given by $v_R(r)=R^{\frac{2}{2-p}} w(r/R)$, for any $R >0$. Consequently, the support of $v_R$ is the ball of radius $2^{\frac 1 N} R$, consistent with Remark~\ref{scaling}.
\end{remark}

% \subsection{A nodal solution}
% Let
% \begin{align*}
% u(x):=u_1(x)\chi_{\{|x|<r_1\}}+u_2(x)\chi_{\{r_1<|x|<1\}}+u_3(x)\chi_{\{1<|x|<r_2\}},    
% \end{align*}
% where
% \begin{align*}
%     u_1(x)=c_1-|x|^2/2,\quad 
%     u_2(x)=(|x|-c_2)^2/2-c_3,\quad
%     u_3(x)=-(|x|-r_2)^2/2,
% \end{align*}
% $r_1=\frac{1}{7} \left(4-\sqrt{2}\right),$ 
% $c_1=\frac{1}{49} \left(9-4\sqrt{2}\right),$
% $c_2=\frac{2}{7} \left(4-\sqrt{2}\right)$, 
% $c_3=\frac{1}{49} \left(9-4\sqrt{2}\right),$
% $r_2=\frac{2}{7} \left(3+\sqrt{2}\right)$. Then $u$ is a nodal solution of the equation
% \begin{align*}
%     -u'' = Q_{B_1}\sgn(u)\qquad \text{ in }\r.
% \end{align*}
% By scaling, see Remark~\ref{scaling}, we can build from $u$ an explicit nodal solution for $\Omega= B_R$ with any $R >0$. 

% \begin{figure}[ht!]
%     \centering
%     \includegraphics[width=0.5\linewidth]{comp.pdf}
%     \caption{A comparison between a nonnegative and a nodal solution of $\eqref{sublinear}$ with  $N=1$, $\Omega=[-1,1],$ and $p=1$. Note that the supports are different.}
% \end{figure}

\section*{Acknowledgments}
We thank the anonymous referees for their comments and suggestions that substantially improved this paper.  We also thank Gabrielle Nornberg, Giorgio Tortone, and Tobias Weth for helpful discussions. 

Delia Schiera is partially supported by the Portuguese government through FCT Fundação para a Ciência e a Tecnologia, I.P., under the projects UID/4459/2025 (CAMGSD), 2023.17881.ICDT with DOI  10.54499/2023.17881.ICDT (project SHADE), and 2024.14494.PEX with DOI 10.54499/2024.14494.PEX (project ASSO).
She is also supported by GNAMPA (Gruppo Nazionale per l’Analisi, Probabilità e le loro Applicazioni) – INdAM (Istituto
Nazionale di Alta Matematica), through the project ‘Critical and limiting phenomena in nonlinear elliptic
systems’, CUP E5324001950001, and by FCT 
under the Scientific Employment Stimulus - Individual Call (CEEC Individual), DOI 10.54499/2020.02540.CEECIND/CP1587/CT0008.

This work was developed while Delia Schiera was visiting the Universidad Nacional Autónoma de México: she gratefully acknowledges the financial support of FCT through the FCT Mobility grant with number FCT/Mobility/1304796440/2024-25, as well as the warm hospitality she received during her stay. 

A. Saldaña is supported by SECIHTI grant CBF2023-2024-116 (Mexico) and by UNAM-DGAPA-PAPIIT grant IN102925 (Mexico).

%\bibliographystyle{amsplain}
%\bibliography{refs.bib}

\noindent\textbf{Mónica Clapp}\\
Instituto de Matemáticas \\
Universidad Nacional Autónoma de México \\
Campus Juriquilla \\
76230 Querétaro, Qro., Mexico \\
\texttt{monica.clapp@im.unam.mx}

\medskip 

\noindent\textbf{Alberto Saldaña}\\
Instituto de Matemáticas \\
Universidad Nacional Autónoma de México \\
Circuito Exterior, Ciudad Universitaria \\
04510 Coyoacán, Ciudad de México, Mexico \\
\texttt{alberto.saldana@im.unam.mx}

\medskip 

\noindent\textbf{Delia Schiera}\\
CAMGSD - Centro de An\'alise Matem\'atica, Geometria e Sistemas Din\^amicos\\
Departamento de Matem\'atica do Instituto Superior T\'ecnico\\
Universidade de Lisboa\\
1049-001 Lisboa, Portugal\\
\texttt{delia.schiera@tecnico.ulisboa.pt}

\end{document}